\documentclass[11pt, a4paper]{amsart}
\usepackage[shortlabels]{enumitem}
\usepackage{amsmath, amsfonts, amscd, mathrsfs, mathtools,graphicx}
\usepackage{latexsym,amssymb,amsthm}
\usepackage{color}
\usepackage{dsfont}
\usepackage{fancyhdr}
\usepackage{tikz}
\usepackage{hyperref}
\usepackage{float}
\usepackage{enumitem}
\usepackage[margin=2.4cm]{geometry}
\usepackage{multicol}
\usepackage{etoolbox}

\AfterEndEnvironment{proof}{  \par\addvspace{12pt}}

\AfterEndEnvironment{rk}{  \par\addvspace{12pt}}

\AfterEndEnvironment{rks}{  \par\addvspace{12pt}}

\newcommand\C{\mathcal{C}}

\newcommand\fix{\mathrm{fix}}

\newcommand\lcm{\mathrm{lcm}}

\newcommand{\comment}[1]{}

 \newcommand{\widesim}{
  \mathrel{{\scalebox{1.5}[1]{$\sim$}}}
}
\newcommand{\reducesize}[2]{  \mathbin{    \ooalign{      \raisebox       {.4ex}          {$#1\widesim$}      \cr       \hidewidth      \raisebox        {-.6ex}        {\scalebox          {.75}          {$#1#2$}        }      \hidewidth    }  }}

\newcommand{\nreducesize}[2]{  \mathbin{    \ooalign{      \raisebox       {.4ex}          {$#1\not\widesim$}      \cr       \hidewidth      \raisebox        {-.6ex}        {\scalebox          {.75}          {$#1#2$}        }      \hidewidth    }  }}

\newcommand{\stb}[1]{\mathpalette\reducesize{#1}}
\newcommand{\nstb}[1]{\mathpalette\nreducesize{#1}}
\newcommand{\magma}{{\sc Magma}}
\newcommand{\gap}{{\sf GAP}}
\def\cn{\mathord{{\!\:{:}\:\!}}}
\usepackage{array}   \newcolumntype{L}{>{$}l<{$}}
\let\oldsection\section
\newcommand\boldsection[1]{\oldsection{\bf #1}}
\newcommand\starsection[1]{\oldsection*{\bf #1}}
\makeatletter
\renewcommand\section{\@ifstar\starsection\boldsection}
\makeatother

\newtheoremstyle{algorithm}
  {12pt}		    {12pt}    {\tt}    {\parindent}       {\bf}    {. }      {\newline}      {}     \theoremstyle{algorithm}

\newtheoremstyle{theorem}
  {12pt}		    {12pt}    {\sl}    {}       {\bf}    {. }      { }      {}     \theoremstyle{theorem}
\newtheorem{thm}{Theorem}[section]  \newtheorem{lemma}[thm]{Lemma}     \newtheorem{cor}[thm]{Corollary}

\newtheorem{prop}[thm]{Proposition}

\newtheoremstyle{definition}
  {12pt}		    {12pt}    {}    {}       {\bf}    {. }      { }      {}     \theoremstyle{definition}

\newcommand\rk{\noindent{\sc Remark.} }

\renewcommand{\proofname}{Proof}

\makeatletter
\renewenvironment{proof}[1][\proofname]{\par
  \pushQED{\qed}  \normalfont \partopsep=\z@skip \topsep=\z@skip
  \trivlist
  \item[\hskip\labelsep
        \scshape
    #1\@addpunct{.}]\ignorespaces
}{  \popQED\endtrivlist\@endpefalse
}
\makeatother

\usepackage{graphicx}

\begin{document}
\title[Binary sporadic groups]{The binary actions of\\ sporadic groups}
\author{Coen del Valle \& Nick Gill}
\date{\today}
\thanks{This research is supported by the Engineering and Physical Sciences Research Council (EPSRC) [grant number EP/Z534742/1].}
\maketitle
\vspace{-0.5cm}\begin{abstract} An action of a group is \emph{binary} if it induces the group of automorphisms of some homogeneous edge-coloured directed graph. In this paper we study the quasisimple groups $G$ with $G/Z(G)$ a sporadic simple group---we produce a complete classification of their binary actions.\end{abstract}

\section{Introduction}
Let $G\leq\mathrm{Sym}(\Omega)$ be a permutation group. Informally, the \emph{relational complexity} of $G$ is the minimum integer $k\geq 2$ such that the orbits of $G$ in its induced action on $\Omega^k$ determine the orbits of $G$ on $\Omega^n$ for every $n\geq k$---this notion will be defined formally in Section~\ref{sec:machine}. Relational complexity has its origins in model theory where it is connected to Lachlan's theory of finite homogeneous structures. Specifically, Cherlin observed that the relational complexity of $G$ is the least integer $k\geq 2$ such that $G$ acts naturally as the automorphism group of some homogenous relational structure whose relations are $k$-ary \cite{cherlin2000}.

In this paper, we are interested in the case that the relational complexity of $G$ is 2; we call such permutation groups (and their associated action) \emph{binary}. In this case, the permutation group $G$ can be represented naturally as the automorphism group of a finite homogeneous edge-coloured directed graph. When restricting to the uncoloured case, such graphs were classified by Lachlan in 1982~\cite{Lach82}, but the coloured case is vastly more complicated. Indeed, to date there is no classification of such graphs, even when restricting to a fixed number $c\geq 2$ of colours---thanks to the above observation of Cherlin, a classification of the finite binary permutation groups would all but resolve this problem.

Binary groups have already been classified in the case that $G$ is primitive~\cite{GLS22}, resolving a long-standing conjecture of Cherlin~\cite{cherlin2000}. Recently, a project has been initiated with the aim of classifying \emph{all} binary permutation groups; the natural starting point being the quasisimple groups. There has already been promising progress: the binary actions of alternating groups have been classified~\cite{GG23}, and there are partial classifications in the case that $G$ is either a simple group of Lie type in characteristic 2~\cite{GGL25} or a simple group with a unique conjugacy class of involutions~\cite{GG25}. Moreover, a forthcoming paper of the authors with Liebeck~\cite{dVGL26} classifies the transitive binary actions in the case that $G$ is a quasisimple group of Lie type of (untwisted) rank at most 2.

We prove a complete classification of the binary actions of the quasisimple sporadic groups. Note that we include the Tits group ${}^2F_4(2)'$ in our consideration of sporadic groups. First we state the classification in the case where $G$ is transitive.

\begin{thm}\label{thm:maintrans}
    Let $G$ be a quasisimple group with $G/Z(G)$ a sporadic simple group and suppose that $G$ acts transitively and faithfully on a set with point stabiliser $H$. The action of $G$ is binary if and only if either $H=1$ or $G$ appears in Table~\ref{tbl:transall} with $H$ cyclic of prime order generated by an element from one of the given conjugacy classes.
\end{thm}

\begin{table}[h]
\begin{tabular}{|ll|ll|ll|ll|}
\hline
$G$                & $\C$ & $G$                  & $\C$   & $G$                  & $\C$  & $G$                  & $\C$  \\ \hline
$\mathrm{Co}_2$    & 2A   & $\mathrm{McL}$       & 3A     & $2.\mathrm{Co}_1$    & 3a    & $6.\mathrm{Fi}_{22}$ & 2b, 2c \\ \hline
$\mathrm{Fi}_{22}$ & 2A   & $\mathrm{Ly}$        & 3A     & $2.\mathbb{B}$       & 2b    &                      &       \\ \hline
$\mathrm{Fi}_{23}$ & 2A   & $2.\mathrm{J}_2$     & 2c, 3a  & $3.\mathrm{McL}$     & 3c, 3d &                      &       \\ \hline
$\mathbb{B}$       & 2A   & $2.\mathrm{Fi}_{22}$ & 2b, 2c & $3.\mathrm{Fi}_{22}$ & 2a    &                      &       \\ \hline
\end{tabular}
\caption{Pairs $(G,\C)$ where $G$ is a sporadic quasisimple group and $\C$ is a conjugacy class such that the transitive action of $G$ with stabiliser $\langle g\rangle$ is binary whenever $g\in \C$. Classes in simple groups are labelled in $\mathbb{ATLAS}$ notation~\cite{atlas}, and classes in the covers are labelled as in the \gap~Character Table Library~\cite{CTblLib1.3.11}.}
\label{tbl:transall}
\end{table}

When it comes to classifying intransitive binary actions, the key fact is this: If the action of a group $G$ on $\Omega$ is binary, then the action of $G$ on each of its orbits is also binary. Given Theorem~\ref{thm:maintrans}, our job is, then, to deduce, for each quasisimple sporadic group, which combinations of transitive binary actions can form orbits in a putative intransitive binary action. The coming theorem is the first result in the literature completely classifying intransitive binary actions in a `non-trivial' setting---this turns out to require a remarkably delicate analysis. To this end, we build up a significant collection of tools in Section~\ref{sec:intrans}, the utility of which should extend to intransitive actions well beyond sporadic groups.

Before stating the result we must introduce some notation. Given a $G$-set, $\Omega$, define $\mathcal{S}(\Omega)$ to be the set of all conjugacy classes of stabilisers of points of $\Omega$. As an abuse of notation, we will write the elements of $\mathcal{S}(\Omega)$ as conjugacy class representatives here, rather than the conjugacy classes themselves---thus two sets may appear different depending on the choice of representative, but are meant to be taken under element-wise equivalence. Our main result is the following.
\begin{thm}\label{thm:main}
Let $G$ be a quasisimple group with $G/Z(G)$ a sporadic simple group. The action of $G$ on the set $\Omega$ is binary if and only if one of the following holds
\begin{enumerate}[(i)]
    \item $\mathcal{S}(\Omega)\subseteq\{1,C, G \mid C\leq Z(G)\}$;
    \item $G\in \{\mathrm{Co}_2,\mathrm{Fi}_{22},\mathrm{Fi}_{23},\mathbb{B}\}$ and $\mathcal{S}(\Omega)\subseteq\{ 1 ,\langle g\rangle , G \}$ for some $g\in\mathrm{2A}$;
    \item $G\in\{\mathrm{McL},\mathrm{Ly}\}$ and $\mathcal{S}(\Omega)\subseteq\{ 1 ,\langle h\rangle , G \}$ for some $h\in \mathrm{3A}$; or
    \item $G$ appears in Table~\ref{tbl:intransall} and $\mathcal{S}(\Omega)\setminus\{1,G\} $ is a subset of one of the relevant maximal binary stabiliser sets displayed.
\end{enumerate}
\end{thm}
\begin{rks}
    Case (i) accounts for the vast majority of binary actions of sporadic quasisimple groups. Moreover, all sporadic simple groups are covered by cases (i)--(iii).
\end{rks}

\begin{table}[h]
\resizebox{0.98\textwidth}{!}{\begin{tabular}{|l|l|l|l|}\hline
$G$               & $\mathcal{S}(\Omega)\setminus\{1,G\}$                             & $G$                  & $\mathcal{S}(\Omega)\setminus\{1,G\}$                                                                                                               \\\hline
$2.\mathrm{J}_2$  & $\{ Z(G) ,\langle g\rangle  ,\langle h\rangle \}$; $g\in\mathrm{2c}$, $h\in\mathrm{3a}$          & $2.\mathrm{Fi}_{22}$ & $\{  Z(G) ,\langle g\rangle ,\langle g,z\rangle \}$,                                                                                   \\\cline{1-2}
$2.\mathrm{Co}_1$ & $\{ Z(G) ,\langle h\rangle \}$; $h\in\mathrm{3a}$                              &                      & $\{  Z(G) ,\langle gz\rangle ,\langle g,z\rangle  \}$, or                                                                                \\\cline{1-2}
$2.\mathbb{B}$    & $\{  Z(G) ,\langle g,z\rangle  \}$, or                    &                      & $\{\langle g\rangle ,\langle gz\rangle \}$; $g\in\mathrm{2b}$                                                                                                  \\\cline{3-4}
                  & $\{ \langle g\rangle \}$; $g\in \mathrm{2b}$                                     & $3.\mathrm{Fi}_{22}$ & $\{Z(G),\langle g\rangle,\langle g,z\rangle\}$; $g\in\mathrm{2a}$                                                                                                                                               \\\hline
$3.\mathrm{McL}$  & $\{ Z(G) ,\langle h\rangle ,\langle h,z\rangle  \}$, or & $6.\mathrm{Fi}_{22}$ & $\{\langle g,z\rangle,\langle g,z^3\rangle,\langle g\rangle,\langle gz^2\rangle,\langle z\rangle,\langle z^2\rangle,\langle z^3\rangle\}$,     \\
                  & $\{\langle hz\rangle  \}$;  $h\in \mathrm{3c}$                                 &                      & $\{\langle g,z\rangle,\langle g,z^3\rangle,\langle gz\rangle,\langle gz^3\rangle,\langle z\rangle,\langle z^2\rangle,\langle z^3\rangle\}$, or \\                  &                                                              &                      & $\{\langle g\rangle,\langle gz\rangle,\langle gz^2\rangle,\langle gz^3\rangle,\langle z^2\rangle\}$; $g\in\mathrm{2b}$                                          \\\hline                                                                                                                                              
\end{tabular}}
\caption{Sporadic quasisimple groups $G$ and their maximal binary stabiliser sets. An action of $G$ on $\Omega$ is binary if and only if $\mathcal{S}(\Omega)\setminus\{1,G\}$ is a subset of at least one of the relevant sets. We use \gap~Character Table Library~\cite{CTblLib1.3.11} class labels, and $z$ is a generator of $Z(G)$.}
\label{tbl:intransall}
\end{table}

Theorem~\ref{thm:main} is proved using a combination of computational and theoretical techniques, and crucially relies on a reduction proved in work of the first author with Craven and Parker~\cite{CDP}. As part of this work, general use \magma~\cite{magma} programs have been developed, which can be accessed from the GitHub repository~\cite{git}, along with explanations and worked examples found in the supplementary document~\cite{comps}.

The structure of this paper is straightforward. In Section~\ref{sec:machine} we collect general structural results about transitive binary actions, before applying these results in Section~\ref{sec:class}, which is dedicated to proving Theorem~\ref{thm:maintrans}. Section~\ref{sec:intrans} tackles the intransitive case, and is split into three subsections; in the first, we develop the general theory of intransitive binary actions, in the second we restrict our attention to the case $Z(G)$ has order 6, and in the final we prove Theorem~\ref{thm:main}.

\section{General machinery}\label{sec:machine}

In this section we develop the theory and key tools which power the proof of the main result. We begin by formally defining relational complexity. 

Suppose that $G$ acts on a finite set $\Omega$, and let $I=(\alpha_i)_{1\leq i\leq n}$ and $J=(\beta_i)_{1\leq i\leq n}$ be $n$-tuples of points of $\Omega$. For an integer $k\leq n$, we say that $I\stb{k}J$ if for every $S\subseteq\{1,2,\dots,n\}$ with $|S|=k$ there exists some $g\in G$ such that $(\alpha_i)_{i\in S}^g=(\beta_i)_{i\in S}$. It is easy to see that $\stb{k}$ defines an equivalence relation. Given such $G$ and $\Omega$, the \emph{relational complexity}, written $\mathrm{RC}(G,\Omega)$, is the least integer $k\geq 2$ such that for any $n\geq k$ and $I,J\in\Omega^n$ if $I\stb{k}J$, then $I\stb{n}J$. When $\mathrm{RC}(G,\Omega)=2$ we say that this action is \emph{binary}. It is straightforward to see that every non-trivial finite group admits at least two transitive binary actions---the regular action, and the action on a single point---we refer to these two actions as the \emph{trivial binary actions}.

For the remainder, $G$ is a finite group. We proceed by reproducing some earlier results which have recently become standard in the study of binary actions---the first is little more than an observation.

\begin{lemma}\label{lem:sub}
    Let $H<M<G$ and suppose that the action of $G$ on the coset space $(G:H)$ is binary. Then the action of $M$ on $(M:H)$ is binary.
\end{lemma}
\begin{rk}   Lemma~\ref{lem:sub} is a special case of~\cite[Lemma 1.7.2]{GLS22}. 
\end{rk}

The next result gives some equivalent conditions useful for identifying binary
actions.
\begin{lemma}[{\cite[Lemma 2.4]{GG23}}]\label{lem:basiccriteria}
    Let $G$ act on $\Omega$. The following are equivalent:\begin{enumerate}
        \item[1.] There exist $I,J\in \Omega^3$ such that $I\stb{2} J$ but $I\nstb{3} J$.
        \item[2.] There exist elements $h_1,h_2,h_3$ of point stabilisers $H_1,H_2,H_3$, respectively, where\begin{enumerate}
            \item $h_1h_2h_3=1$; and
            \item there do not exist $h_2'\in H_1\cap H_2$ and $h_3'\in H_1\cap H_3$ with $h_1h_2'h_3'=1$.
        \end{enumerate}
        \item[3.] There exist point stabilisers $H_1,H_2,H_3$, such that $H_1\cap (H_2\cdot H_3)\not\subseteq H_1\cap((H_1\cap H_2)\cdot H_3).$
    \end{enumerate}
\end{lemma}
A subgroup $H<G$ is a \emph{TI-subgroup} if $H\cap H^g=1$ for all $g\in G\setminus N_G(H)$.

\begin{lemma}[{\cite[Lemma 2.3]{GG25}}]\label{lem:basiccriteriaTI}
    The action of $G$ on the cosets of a TI-subgroup $H$ is binary if and only if ${H_1\cap (H_2 \cdot H_3)= \{1\}}$ for every triple $(H_1,H_2,H_3)$ of distinct conjugates of $H$.
\end{lemma}
Sometimes we will not be lucky enough to have a TI-subgroup $H$ to test, but provided that $H$ is small (and so has many conjugates with which it intersects trivially), the following lemma gives another useful and efficient test.
\begin{lemma}\label{lem:efficienttest}
    Let $G$ act on the coset space $(G:H)$ for some $H<G$. If there exist conjugates $H_2,H_3$ of $H$ and elements $1\ne h_1\in H$, $h_2\in H_2$ and $h_3\in H_3$ such that $H\cap H_2=H\cap H_3=1$ and $h_1h_2h_3=1$, then the action of $G$ on $(G:H)$ is not binary.
\end{lemma}
\begin{proof}
    We use Statement 2 of Lemma~\ref{lem:basiccriteria}: if $h_2'\in H\cap H_2$ and $h_3'\in H\cap H_3$, then $h_2'=h_3'=1$, by assumption. In particular, $h_1h_2'h_3'=h_1\ne 1$, so the conditions of Statement 2 are satisfied, as needed.
\end{proof}
  
\begin{lemma}[{\cite[Lemma 1.8.1]{GLS22}}]\label{lem:frob}
    Let $G$ be a Frobenius permutation group on $\Omega$ with Frobenius complement of order $m>2$. Then the action of $G$ is non-binary.
\end{lemma}

Of course, for the sporadic groups we will never find ourselves in the situation of Lemma~\ref{lem:frob}, however ~\cite[Lemma 1.7.1]{GLS22} allows us to instead study the actions of point stabilisers. We reproduce (a special case of) the result here.
\begin{lemma}\label{lem:suborb}
    Let $G$ be a transitive permutation group with point stabiliser $H$. If there exists an $H$-orbit on which the action of $H$ is non-binary, then $G$ is non-binary.
\end{lemma}
Before our next key definition, recall that for $x\in G$ of order $p$, the \emph{rational class} of $x$ is the set $\{(x^i)^g\mid 1\leq i\leq p-1,g\in G\}$ of all conjugates of non-trivial powers of $x$. 

We turn next to a graph which was first defined in~\cite{GG23}. This graph---dubbed the \emph{Gill--Guillot graph} in~\cite{CDP}---turns out to be very useful for investigating binary actions. We define it as follows: given a conjugacy class or rational class $\mathcal{C}\subseteq G$, the graph $\Gamma(\mathcal{C})$ has vertex set $\mathcal{C}$, where $x,y\in\mathcal{C}$ are adjacent exactly when $x$ and $y$ commute and $\{yx^{-1},xy^{-1}\}\cap\mathcal{C}\neq\emptyset$. Now, given an element $g\in\mathcal{C}$, we define the \emph{component group} of $g$ in $\mathcal{C}$, to be the group $\Delta(g,\mathcal{C})$ generated by the connected component of $\Gamma(\mathcal{C})$ containing $g$. When clear from context we shall write $\Delta(g)=\Delta(g,\mathcal{C})$. 

The component groups of the Gill--Guillot graphs $\Gamma(\mathcal{C})$ are classified for all conjugacy classes of prime order elements of sporadic quasisimple groups in~\cite{CDP}. Our classification of binary actions relies heavily on these results, with the connection to binary actions given by the next result.

 \begin{lemma}[{\cite[Corollary 2.13]{GG23}}]\label{lem:conjconn}
     Let $G$ act on $\Omega:=(G:H)$ for some $H<G$, and assume that the action is binary. Suppose that $g\in H$ is an element of prime order $p$ and $\C$ is the conjugacy class (resp. rational class) of $g$. Then either $\Delta(g,\C)\leq H$, or there exists an element $h$ of order $p$ such that $\mathrm{fix}(g)\subsetneq\fix(h)$ and $\Delta(h,\mathcal{D})\leq H$ where $\mathcal{D}$ is the conjugacy class (resp. rational class) of $h$.
 \end{lemma}

\begin{proof}
Suppose that $\Delta(g)\not\leq H$. Then there exists $g_1\in \mathcal{C}\cap H$ and $g_2\in\mathcal{C}\setminus H$ such that $g_1$ and $g_2$ are adjacent in $\Gamma(\mathcal{C})$. Write $H=G_\alpha$ for some $\alpha\in\Omega$.

Set $\Lambda=\{\lambda_1,\dots,\lambda_k\}=\fix(g_1)\cup\fix(g_2)\cup\fix(g_1g_2^{-1})$. Since $g_1$ and $g_2$ commute, both preserve $\Lambda$, and so $K:=\langle g_1,g_2\rangle$ acts on $\Lambda$. Next,  let $\Lambda_0=\Lambda\setminus(\fix(g_1)\cup\fix(g_2))$, and observe that $\Lambda_0$ is non-empty: Indeed, $|\fix(g_1)|=|\fix(g_2)|=|\fix(g_1g_2^{-1})|$, and $$\fix(g_1)\cap\fix(g_2)=\fix(K)=\fix(g_1)\cap\fix(g_1g_2^{-1})=\fix(g_2)\cap\fix(g_1g_2^{-1}),$$ whence $$|\Lambda_0|=|\fix(g_1g_2^{-1})|-|\fix(K)|=|\fix(g_1)|-|\fix(K)|.$$ Now, $\fix(K)\subseteq \fix(g_1)$, but $\alpha\in \fix(g_1)$ whereas $\alpha\not\in\fix(K)$ (since $g_2\not\in H$), thus $|\Lambda_0|>0$. 

The group $K$ acts on $\Lambda_0$, with $g_1$ and $g_2$ inducing the same non-trivial permutation $\tau_0\in\mathrm{Sym}(\Lambda_0)$. Write $\tau\in\mathrm{Sym}(\Lambda)$ so that $\tau|_{\Lambda_0}=\tau_0$ and $\lambda_i^\tau=\lambda_i$ whenever $\lambda_i\in\Lambda\setminus\Lambda_0$, and define $I=(\lambda_1,\dots,\lambda_k)$ and $J=(\lambda_1^\tau,\dots,\lambda_k^\tau)$. Clearly, $I\stb{2} J$, so since the action of $G$ is binary there is some $r\in G$ of order $p^m$ for some $m\geq 1$ (replacing $r$ with one of its powers, if necessary) such that $r|_\Lambda=\tau$. In particular, $\fix(g_1)\subsetneq\fix(r)$, and so by conjugating and taking powers if necessary we get an element $r'\in G$ of order $p$ such that $\fix(g)\subsetneq\fix(r')$.

If $\Delta(r')\leq H$, then we are done. If not we repeat the argument with $r'$ in place of $g$. Since the size of the set of fixed points strictly increases, and $\Omega$ is finite, the process eventually terminates with $h\in G$ satisfying $\Delta(h)\leq H$, as desired.
\end{proof}

Suppose that $G$ is acting on a set $\Omega$ and that $p$ is a prime dividing $|G|$. We say that $g\in G$ is of \emph{maximal $p$-fixity} if $g$ has order $p$ and $\mathrm{fix}(g)$ is not properly contained in $\mathrm{fix}(h)$ for any $h\in G$ of order $p$.

The utility of Lemma~\ref{lem:conjconn} can be nicely demonstrated in the following situation: suppose $\mathcal{C}$ is a conjugacy class of elements of $G$ of order $p$ such that $\Gamma(\mathcal{C})$ is connected and such that $H\cap\mathcal{C}\ne\emptyset$. If the elements of $\mathcal{C}$ are of maximal $p$-fixity in the binary action of $G$ on $\Omega=(G:H)$, then Lemma~\ref{lem:conjconn} implies that $\langle\mathcal{C}\rangle$ is a non-trivial normal subgroup of $G$ which is also a subgroup of $H$. In particular the action of $G$ on $\Omega$ is not faithful.

\begin{lemma}[{\cite[Lemma 2.6]{GGL25}}]\label{lem:maxfixdelta}
    Let $H<G$ be such that the action of $G$ on $\Omega:=(G:H)$ is binary, and let $g\in H$ be an element of maximal $p$-fixity. Then every element of $\Delta(g)$ of order $p$ is an element of maximal $p$-fixity.
\end{lemma}
\begin{proof}
    Fix $h\in\Delta(g)$ with order $p$ and let $\alpha\in \fix(g)$. Then $\Delta(g)\leq G_\alpha$ by Lemma~\ref{lem:conjconn}, and so $h\in G_\alpha$. Therefore, since $\alpha$ was arbitrary, $\fix(g)\subseteq\fix(h)$, whence $h$ is an element of maximal $p$-fixity, as desired.
\end{proof}

When $G$ has multiple classes of $p$-elements, the next lemma is often useful for restricting the structure of the stabiliser in a transitive binary action. Before stating it we define a family of groups. Consider the action of $G$ on the cosets of some subgroup $H<G$. Given $g,h\in H$, we define $F_0(g,h)=\langle g\rangle$, and for $i\in\mathbb{Z}_{\geq 0}$ we define $$F_{i+1}(g,h)=\left\langle h^c \mid c \in G, g^c\in F_j(g,h)\text{ for some } j\in\{0,\dots,i\}\right\rangle.$$ Finally, we define $F_\infty(g,h)=\left\langle\bigcup_{i\geq0} F_i(g,h)\right\rangle.$

\begin{lemma}\label{lem:Finf}
    Let $H<G$ and let $G$ act on $\Omega=(G:H)$. Suppose that $g,h\in H$ are such that $\mathrm{fix}(g)\subseteq \mathrm{fix}(h)$. Then $F_\infty(g,h)\leq H.$
\end{lemma}
\begin{proof}
We show by strong induction that $F_i(g,h)\leq H$ for each $i$, from which the result follows. It is clear that $F_0(g,h)\leq H$, so assume that $i\geq 0$ is such that $F_j(g,h)\leq H$ for all $j\in\{0,\dots,i\}$. Let $c\in G$ such that $g^c\in F_j(g,h)$ for some $j\leq i$, and let $\alpha\in\Omega$ such that $G_\alpha=H$. Then by the inductive hypothesis $g^c\in H=G_{\alpha}$, so $g\in G_{\alpha^{c^{-1}}}$---that is, $\alpha^{c^{-1}}\in\mathrm{fix}(g)\subseteq\mathrm{fix}(h)$, whence $h\in G_{\alpha^{c^{-1}}}$, and so $h^c\in G_\alpha=H$, as was to be shown.
\end{proof}
\begin{rk}
    Note, in particular, that $\langle h^c \mid c\in N_G(\langle g\rangle)\rangle\leq F_\infty(g,h).$
\end{rk}
  
\begin{lemma}\label{lem:normalsub}
 Let $G$ be a finite quasisimple group, and let $H<G$ with $Z(G)\cap H=1$. Suppose that the action of $G$ on $\Omega=(G:H)$ is binary, and let $p\mid |H|$ be prime. Suppose that $G$ has a unique rational class $\mathcal{C}$ of non-central elements of order $p$ such that $\Gamma(\mathcal{C})$ is disconnected.  Then for every element $g$ of order $p$ in $H\setminus\mathcal{C}$ there is a pair $(K,L)$ of subgroups of $G$ such that $K$ is a subgroup of $H$ intersecting $\mathcal{C}$ non-trivially, and satisfying $\langle g\rangle<K\trianglelefteq L$, and $N_G(\langle g\rangle)\leq L$. Thus, if $N_G(\langle g\rangle)$ is maximal in $G$, then $H$ contains a normal subgroup of $N_G(\langle g\rangle)$ properly containing $\langle g\rangle$.
\end{lemma}

\begin{proof}
Let $g$ be as in the statement. By Lemma~\ref{lem:conjconn}, since $G$ is quasisimple there must be some $h\in H\cap \mathcal{C}$ such that $\fix(g)\subseteq\fix(h)$. Now, by Lemma~\ref{lem:Finf}, $$K:=\langle g,h^c \mid c\in N_G(\langle g\rangle)\rangle\leq H.$$
Clearly $N_G(\langle g\rangle)$ permutes the above generators for $K$, in particular $N_G(\langle g\rangle)$ normalises $K$. Therefore, $N_G(\langle g\rangle)\leq N_G(K)$; setting $L=N_G(K)$ yields the desired result.\end{proof}

It turns out that character theory gives us powerful tools for both understanding the graph $\Gamma(\mathcal{C})$ (see~\cite{CDP} for more discussion) and for detecting binary actions: Given a group $G$ with conjugacy classes $\mathcal{C}_1,\mathcal{C}_2,\mathcal{C}_3$, the \emph{class multiplication coefficient} $n_G(\mathcal{C}_2,\mathcal{C}_3,\mathcal{C}_1)$ is the number of ways a fixed element $t\in \mathcal{C}_1$ can be written as a product $xy$ where $x\in\mathcal{C}_2$ and $y\in \mathcal{C}_3$. It is a standard exercise in introductory character theory (see e.g.~\cite[Problems (3.9)]{I94}) that $$n_G({\mathcal{C}_2,\mathcal{C}_3,\mathcal{C}_1})=\frac{|\mathcal{C}_2||\mathcal{C}_3|}{|G|}\sum_{\chi\in\mathrm{Irr}(G)}\frac{\chi(g_2)\chi(g_3)\overline{\chi(g_1)}}{\chi(1)},$$ where $g_1,$ $g_2,$ and $g_3$ are representatives of the classes $\mathcal{C}_1$, $\mathcal{C}_2$, and $\mathcal{C}_3$, respectively. In particular, when the character table of $G$ is known (as is the case for the sporadic groups), the quantities $n_G(\mathcal{C}_2,\mathcal{C}_3,\mathcal{C}_1)$ can easily be computed.

In this paper, we primarily use the class multiplication coefficients as a way to recognise transitive binary actions with stabiliser of prime order. We write $n_G(\mathcal{C}):=n_G(\mathcal{C},\mathcal{C},\mathcal{C})$.
\begin{lemma}\label{lem:cycbintest}
    Let $G$ be a finite group, let $\mathcal{C}$ be a rational class of elements of order $p$ written as a union $\bigsqcup_{i=1}^m\C_i$ of conjugacy classes. Let $g\in\mathcal{C}$. If the action of $G$ on $G/\langle g\rangle$ is binary then $\sum_{(i,j)\in[m]^2}n_G(\mathcal{C}_i,\mathcal{C}_j,\mathcal{C}_1)=p-2$. On the other hand, if $\sum_{(i,j)\in[m]^2}n_G(\mathcal{C}_i,\mathcal{C}_j,\mathcal{C}_k)=p-2$ for all $k\in[m]$ then the action of $G$ on $G/\langle g\rangle$ is binary.
\end{lemma}
\begin{proof}
        Without loss of generality we may assume that $g\in\C_1$. Let $$S=\{(g^2,g^{p-1}),(g^3,g^{p-2}),\dots,(g^{p-1},g^2)\}.$$ It is clear that $\sum_{(i,j)\in[m]^2}n_G(\mathcal{C}_i,\mathcal{C}_j,\mathcal{C}_1)\geq |S|=p-2$. Suppose that $$\sum_{(i,j)\in[m]^2}n_G(\mathcal{C}_i,\mathcal{C}_j,\mathcal{C}_1)>(p-2).$$ Then there exist $(h,r)\in (\mathcal{C}\times \mathcal{C})\setminus S$ such that $hr=g$, thus $g\in\langle g\rangle\cap(\langle h\rangle\cdot\langle r\rangle)$ so the action is not binary by Lemma~\ref{lem:basiccriteriaTI}. 

 On the other hand, suppose that $\sum_{(i,j)\in[m]^2}n_G(\mathcal{C}_i,\mathcal{C}_j,\mathcal{C}_k)=p-2$ for all $k\in[m]$. Then $S$ is the set of all pairs $(h,r)\in\mathcal{C}\times\mathcal{C}$ such that $hr=g$. It follows that $\langle g\rangle\cap(\langle h\rangle\cdot\langle r\rangle)=\{1\}$ unless $h,r$ are powers of $g$, thus the action is binary by Lemma~\ref{lem:basiccriteriaTI}, as desired.
\end{proof}
In~\cite{CDP}, when it comes to triple covers, the authors only consider the graphs on conjugacy classes rather than rational classes. For the cases for which we need to examine connectivity of the graphs on elements of order three, we will be able to use the next result. 

\begin{lemma}\label{lem:conn3cov}
Suppose that $G$ is a triple cover of a simple group such that $\mathcal{C}$ is a conjugacy class of $G/Z(G)$ lifting to three $G$-conjugacy classes $\mathcal R$, $\mathcal N$, and $\mathcal{N}'=\mathcal N^{-1}$ of elements of order 3, where $\mathcal{R}$ is real, and the other two are not. If $\Gamma(\mathcal{R})$ is connected then $\Gamma(\mathcal{N}\cup\mathcal{N}')$ is connected.
\end{lemma}
\begin{proof}
    Write $Z(G)=\langle z\rangle$ and assume without loss of generality that $\mathcal{N}=\mathcal{R}z$ and $\mathcal{N}'=\mathcal{R}z^2$. Suppose $\{x,y\}$ is an edge in $\Gamma(\mathcal{R})$. Then $[x,y]=1$ and $x^{-1}y\in\mathcal{R}$. Thus $[xz,yz^2]=1$ and $(xz)^{-1}(yz^2)=(x^{-1}y)z\in\mathcal{N}$, and so $xz$ and $yz^2$ are adjacent in $\Gamma(\mathcal{N}\cup\mathcal{N}').$ Consequently, for any $s,t\in\mathcal{R}$, there is a path in $\Gamma(\mathcal{N}\cup\mathcal{N}')$ between $sz$ and either $tz$ or $tz^2$. Therefore, $\Gamma(\mathcal{N}\cup\mathcal{N}')$ has     at most two connected components. Moreover, since every element of $\mathcal{N}$ has neighbours in $\mathcal{N}'$ and vice versa, it follows that $G$ acts transitively on the components. Therefore, by the orbit-stabiliser lemma the stabiliser of a component of $\Gamma(\mathcal{N}\cup\mathcal{N}')$ has index in $G$ at most 2. Therefore, since $G/Z(G)$ is simple, the stabiliser of a component of $\Gamma(\mathcal{N}\cup\mathcal{N}')$ is $G$, and so $\Gamma(\mathcal{N}\cup\mathcal{N}')$ is connected by ~\cite[Lemma 2.2]{CDP}.
\end{proof}

Finally, although not related to binary actions, the following lemma will be useful.
\begin{lemma}\label{lem:SZappl}
    Suppose that $G$ is a finite split extension $H\cn K$ with $H$ a $p$-group. Let $L$ be a $p'$-subgroup of $G$. Then $L$ is conjugate in $G$ to some $M\leq K$.
\end{lemma}
\begin{proof}
    Since $L$ is $p'$ and $H$ is a $p$-group we deduce that $L\cap H=1$. Moreover, $H\trianglelefteq G$ so $J:=HL\leq G$. Thus, using that $L\cap H=1$ we see that $J/H\cong L$ which has order coprime to $|H|$, whence $H$ is a normal subgroup of $J$ with complement $L$.

    Set $M=J\cap K$. Then $HM=H(J\cap K)=J\cap HK=J$, and $M\cap H=J\cap K\cap H=1$. That is, $M$ a complement to $H$ in $J$ and so $M$ and $L$ are conjugate by Schur--Zassenhaus, as was to be shown.
\end{proof}
\section{The classification}\label{sec:class}

From here onward, $G$ is a quasisimple sporadic group acting faithfully on $(G:H)$ for some $H<G$ (so $H\cap Z(G) =1$). We already have enough information to find our complete list of transitive binary actions; the remainder of the section will be devoted towards showing that the list is complete. Throughout, we use $\Bbb{ATLAS}$~\cite{atlas} notation and class names, with classes in the quasisimple covers and maximal subgroups given names corresponding to the labelling of the \gap~Character Table Library~\cite{CTblLib1.3.11}.
\begin{prop}\label{prop:allbins}
    Suppose that the pair $(G,\mathcal{C})$ appears in Table~\ref{tbl:transall}, and that $H$ is generated by an element from $\mathcal{C}$. Then the action of $G$ on $(G:H)$ is binary.
\end{prop}
\begin{proof}
    Let $p$ be the order of the elements of $\mathcal{C}$, and write $\C=\bigsqcup_{i=1}^m\C_i$ a union of conjugacy classes of elements of order $p$. We compute using the \gap~Character Table Library~\cite{CTblLib1.3.11} that $\sum_{(i,j)\in[m]^2}n_G(\mathcal{C}_i,\mathcal{C}_j,\mathcal{C}_1)=p-2$, hence the result follows from Lemma~\ref{lem:cycbintest}, and possibly taking inverses when $p=3$.
\end{proof}
The rest of this section is dedicated to proving the converse of Proposition~\ref{prop:allbins}, and hence proving Theorem~\ref{thm:maintrans}.

The strategy we shall now follow will be similar for all groups $G$; we assume throughout that the action of $G$ on $(G:H)$ is binary. In~\cite{CDP}, the component groups of the relevant graphs $\Gamma(\mathcal{C})$ were classified, and in particular, strong divisibility conditions may be imposed on $|H|$---we reproduce the relevant result below.
\begin{prop}[{\cite[Corollary 1.2]{CDP}}]\label{prop:binarydiv}
Let $p$ be prime such that $(G,p)$ appear in Table~\ref{tbl:forbiddenpbinary}. Then $p\nmid |H|$.
\end{prop}

\begin{table}[h]
    \centering
    \begin{tabular}{|ll|ll|ll|ll|}\hline
$G$            & $p$   & $G$       & $p$     & $G$        & $p$       & $G$         & $p$             \\ \hline
$\mathrm{M}_{11}$      & $2,3$ & $\mathrm{HS}$     & $2,3$   & $\mathrm{O\text{'}N}$      & $2,3$     & $\mathrm{Fi}_{23}$  & $3,5$           \\
$\mathrm{M}_{12}$      & $2,3$ & $\mathrm{J}_3$    & $2$     & $\mathrm{Co}_3$    & $2,3$     & $\mathrm{Co}_1$     & $2,5,7$         \\
$\mathrm{J}_1$         & $2$   & $\mathrm{M}_{24}$ & $2,3$   & $\mathrm{Co}_2$    & --        & $\mathrm{J}_4$      & $2,3$           \\
$\mathrm{M}_{22}$      & $2,3$ & $\mathrm{McL}$    & $2$     & $\mathrm{Fi}_{22}$ & $3,5$     & $\mathrm{Fi}_{24}'$ & $2,3,5,7$       \\
$\mathrm{J}_2$         & $2$   & $\mathrm{He}$     & $2,3,5$ & $\mathrm{HN}$      & $2,3$     & $\mathbb{B}$       & $3,5,7$         \\
$\mathrm{M}_{23}$      & $2,3$ & $\mathrm{Ru}$     & $2,3,5$ & $\mathrm{Ly}$      & $2$       & $\mathbb{M}$       & $2,3,5,7,11,13$ \\
${}^2F_4(2)'$ & $2,3$ & $\mathrm{Suz}$    & $2$     & $\mathrm{Th}$      & $2,3,5,7$ &             &   \\             \hline
\end{tabular}

    \caption{Sporadic groups and primes which do not divide the order of a stabiliser of a faithful transitive binary action.}
    \label{tbl:forbiddenpbinary}
\end{table}
The divisibility conditions of Proposition~\ref{prop:binarydiv} restrict the structure of possible point stabilisers enough so that only a manageable number of subgroups need to be considered---for the smallest sporadics, this can be done entirely by computer.

\begin{prop}\label{prop:1stclas}
    Suppose that $G$ is one of $\mathrm{M}_{11}, \mathrm{M}_{12}, 2.\mathrm{M}_{12},\mathrm{J}_1, \mathrm{M}_{22}, 2.\mathrm{M}_{22}, 3.\mathrm{M}_{22}, 
    4.\mathrm{M}_{22},$ $6.\mathrm{M}_{22}$,
    $12.\mathrm{M}_{22},\mathrm{J}_2,\mathrm{M}_{23},{}^2F_4(2)',\mathrm{HS},2.\mathrm{HS},\mathrm{J}_3, 3.\mathrm{J}_3 ,\mathrm{M}_{24},\mathrm{He},\mathrm{Ru},2.\mathrm{Ru},\mathrm{O\text{'}N},$ or $\mathrm{Co}_3.$ Then $H=1$.
\end{prop}
\begin{proof}
    By Proposition~\ref{prop:binarydiv} together with~\cite[Lemma 2.12, Theorem 4.2, and Table 8]{CDP}, $|H|$ is odd, whence $H$ is solvable. We exploit this in \magma~\cite{magma} by generating the list of all solvable subgroups of $G$, and then filtering out those of order divisible by $p$, where $p$ appears in Table~\ref{tbl:forbiddenpbinary}, as well as those intersecting $Z(G)$ non-trivially. This leaves us with a short list of relatively small candidates for $H$ which we test directly via our method {\tt{NearTITest(G,H)}} (see~\cite{comps} for details), completing the proof.
\end{proof}
\begin{rk}
    The function {\tt{NearTITest(G,H)}} simply runs a random search for witnesses to the conditions of Lemma~\ref{lem:efficienttest}---since the subgroups are all very small, each random search is completed in a matter of seconds.
\end{rk}

For the remainder of this section, rather than attacking several groups simultaneously, we now need to shift to arguments more specialised for each specific group. We begin with $2.\mathrm{J}_2$.
\subsection{$2.\mathrm{J}_2$} Suppose that $G=2.\mathrm{J}_2$.
\begin{prop}\label{prop:2j2}
    Either $H=1$ or $H$ is generated by an element in $\mathrm{2c}\cup\mathrm{3a}$.
\end{prop}

Remember that we are assuming here, and below, that $H\cap Z(G)=1$.

\begin{proof}
    We may assume that $H$ is not generated by an element in $\mathrm{2c}\cup 3a$. Suppose first that $|H|$ is even. By~\cite[Theorem 4.2]{CDP}, the only involution class $\C$ for which $\Gamma(\mathcal{C})$ is disconnected is $\mathrm{2c}$, and hence there is some $h\in H\cap \mathrm{2c}$. The only other non-central involution class is $\mathrm{2b}$; suppose $H\cap\mathrm{2b}\ne\emptyset$. Then by Lemma~\ref{lem:normalsub}, since the centraliser of every non-central involution is maximal in $G$ we deduce that $H$ has a subgroup $K$ satisfying $\langle g\rangle <K\trianglelefteq C_G(g)=2^{1+4}_-\cn2A_5$ where $g\in \mathrm{2b}$. That is, $H$ contains a subgroup $2^{1+4}_-$; we call {\tt Subgroups(G)} and filter out all subgroups of $G$ not containing a subgroup of the form $2^{1+4}_-$ (and those which intersect $Z(G)$ non-trivially). Testing each of the subgroups in this list with {\tt NearTITest()} shows that no such group can be the stabiliser of a binary action, consequently $H\cap\mathrm{2b}=\emptyset$.

    Next, we use our program {\tt FilteredSubgroups()} to generate a list of all subgroups of $G$ which contain an element of $G$-conjugacy class $\mathrm{2c}$ but none from the class $\mathrm{2b}$; we also filter out the cyclic group of order 2. Running {\tt NearTITest()} on each remaining subgroup shows that none yields a binary action, thus $|H|$ is odd.

    Finally, we generate a list of all non-trivial odd order subgroups of $G$ which are not generated by an element of $\mathrm{3a}$. After testing each such subgroup with {\tt NearTITest()} we deduce the desired result.
\end{proof}
\subsection{$3.\mathrm{O\text{'}N}$} Suppose that $G=3.\mathrm{O\text{'}N}$. By~\cite[Theorem 1.1]{CDP} and Lemma~\ref{lem:conjconn}, $|H|$ is indivisible by 2 and 3. In fact, $|H|$ is also indivisible by 7: By~\cite[Proposition 3.2]{CDP} the component group of an element of $\mathrm{7b}$ contains elements of $\mathrm{7a}$---such elements are of maximal 7-fixity by Lemma~\ref{lem:maxfixdelta}, and have connected component group by~\cite[Proposition 3.3]{CDP}, yielding the claim by Lemma~\ref{lem:conjconn}.  This reduces the possible structure enough for us to check the remaining potential stabilisers by computer.

\begin{prop}
    The group $3.\mathrm{O\text{\emph{'}}N}$ has no non-trivial faithful transitive binary actions.
\end{prop}
\begin{proof}
    We run {\tt FilteredSubgroups()} on $G$ to generate lists of all subgroups with order indivisible by each of 2, 3, and 7 but divisible by at least one of 11, 19, or 31. Running {\tt NearTITest()} on each remaining candidate confirms none yield binary actions, so the only possibility that remains is $|H|=5$. We compute using the {\sf GAP} Character Table Library~\cite{CTblLib1.3.11} that $n_G(\mathrm{5a,5a,5a})=753>3$, so the result follows from Lemma~\ref{lem:cycbintest} (using that $G$ has a unique rational class of elements of order 5).
\end{proof}
\subsection{$\mathrm{McL}$}
Suppose that $G/Z(G)=\mathrm{McL}$.
  Let $\varphi:G\to G/Z(G)$ be the quotient map.
\begin{prop}\label{prop:McLclass}
    Either $H=1$ or $H$ is generated by an element $t$ with $\varphi(t)$ in the $G/Z(G)$-conjugacy class $\mathrm{3A}$.
\end{prop}

\begin{proof}
    Since $2\nmid |H|$, we run {\tt SolvableSubgroups()}, and then filter out all subgroups with even order, and all those which intersect $Z(G)$ non-trivially. Next, we remove those which are cyclic and generated by an element with image in $\mathrm{3A}$. We test each of the remaining subgroups using {\tt NearTITest()}, which shows that none yield binary actions, hence the result.
\end{proof}
\rk{The proof above involves two separate computations: one for the simple group $\mathrm{McL}$ and one for the cover $3.\mathrm{McL}$.}
\subsection{$\mathrm{Suz}$}
Suppose next that $G/Z(G)=\mathrm{Suz}$. By Proposition~\ref{prop:binarydiv} and~\cite[Lemma 2.12]{CDP} we know that whenever $2\nmid |Z(G)|$ then $2\nmid |H|$. Moreover, the same holds when $2\mid |Z(G)|$ by \cite[Theorem 1.1]{CDP}. In fact, we can also deduce that $5\nmid |H|$: Indeed, by \cite[Proposition 3.3]{CDP}, for $\mathcal{C}$ a class of elements of order $5$, either $\Gamma(\mathcal{C})$ is connected, or it has component group $\mathrm{J}_2$, which has even order, thus no such element can be present in $H$ by Lemma~\ref{lem:conjconn}. Additionally, for all classes $\C$ of elements of order $p\in\{7,11,13\}$ we compute that $n_G(\C)> p-2$, whence no transitive binary action of $G$ has stabiliser of order $p$ by Lemma~\ref{lem:cycbintest}. These restrictions immediately imply that either $H$ is a 3-group, or $|H|$ is divisible by one of $21,33,39,77,91,$ or 143.
  \begin{prop}\label{prop:suzclass}
    The group $G$ has no non-trivial faithful transitive binary actions.
\end{prop}
\begin{proof}
    Suppose first that $H$ is not a 3-group. Then $|H|$ is divisible by at least one of 21, 33, 39, 77, 91, or 143, and indivisible by $2$ and $5$. We use our program {\tt FilteredSubgroups()} to generate a list (of conjugacy class representatives) of all such subgroups, and test each one with our program {\tt NearTITest()} which confirms that none are the stabiliser of genuine binary actions, thus $H$ is a 3-group.

    From here it is straightforward to get a list of all subgroups of a Sylow 3-subgroup of $G$ using built-in methods in \magma; we do so before filtering out all those intersecting the center non-trivially. We test each remaining candidate directly with {\tt NearTITest()} to deduce the result.
\end{proof}

\subsection{$\mathrm{Co}_2$} The group $G=\mathrm{Co}_2$ poses our first significant challenge, for we have no divisibility restrictions on $H$. Our approach therefore must depend more on structural arguments, rather than brute force computation.

  \begin{prop}\label{prop:Co2no2B}
        The stabiliser $H$ contains no elements of the $G$-conjugacy class $\mathrm{2B}$.
\end{prop}
\begin{proof}
    Suppose there is some $g\in \mathrm{2B}\cap H$. By~\cite[Propositions 3.1, 3.2, and 3.3]{CDP} and Lemma~\ref{lem:conjconn}, the class $\mathrm{2A}$ is the unique class of maximal 2-fixity. In particular, there exists some $h\in\mathrm{2A}\cap H$ with $\fix(g)\subsetneq \fix(h)$. By Lemma~\ref{lem:Finf}, $H$ contains an orbit of $C_G( g)$ on the $G$-conjugacy class $\mathrm{2A}$; we compute each such orbit in \magma~\cite{magma}. We then use our \magma~ function {\tt{Finf()}}~\cite{comps} for each possible orbit representative $h$---the function {\tt{Finf()}} computes a subgroup of $F_\infty(g,h)\leq H$---in each case the output is $G$, contradicting that $H<G$.
\end{proof}

\begin{prop}\label{prop:Co2no3B}
        The stabiliser $H$ contains no elements of the $G$-conjugacy class $\mathrm{3B}$
\end{prop}
\begin{proof}
    Suppose there is some $g\in \mathrm{3B}\cap H$. By~\cite[Proposition 3.1, 3.2, and 3.3]{CDP}, $\Gamma(\mathrm{3B})$ is connected, whence $\mathrm{3A}$ is the unique class of maximal 3-fixity. In particular, there exists some $h\in\mathrm{3A}\cap H$ with $\fix(g)\subsetneq \fix(h)$. By Lemma~\ref{lem:Finf}, $H$ contains an orbit of $N_G(\langle g\rangle)$ on the $G$-conjugacy class $\mathrm{3A}$; we compute each such orbit in \magma~\cite{magma}---this is a memory intensive computation if done directly, we use our function {\tt OrbReps()} to compute orbit representatives in an effort to conserve memory at the cost of time; see~\cite{comps} for details. Each such orbit generates a group containing an element from class $\mathrm{2B}$, contradicting Proposition~\ref{prop:Co2no2B}.
\end{proof}
We can generate a list of all subgroups of $G$ of order $27$, finding that each such subgroup has an element of $G$-conjugacy class $\mathrm{3B}$. Similarly, we generate the list of all subgroups of $G$ of order $9$, and find that the only such subgroup with no $\mathrm{3B}$ is cyclic of order $9$. Consequently we deduce the following.

\begin{cor}\label{cor:Co2cycsyl}
    Let $S$ be a Sylow 3-subgroup of $H$. Then $S$ is cyclic of order at most 9.
\end{cor}

Examining class multiplication coefficients, we observe that the product of any two $\mathrm{2A}$ involutions in $G$ is in $\mathrm{2B}\cup\mathrm{2C}\cup\mathrm{3B}\cup\mathrm{4C}$.  This gives us our first major reduction in the structure of $H$.
\begin{cor}\label{cor:Co2cent2A}
    The stabiliser $H$ contains no element of $G$-conjugacy class $\mathrm{2C}$, and $|H\cap \mathrm{2A}|\leq 1$.     In particular, either $|H|$ is odd, or $H$ has a unique involution. In the latter case this involution is from class $\mathrm{2A}$ and is central in $H$.
\end{cor}
\begin{proof}
    Suppose $x,y\in\mathrm{2A}\cap H$ are distinct. Then $xy\not\in \mathrm{2B}$ by Proposition~\ref{prop:Co2no2B} and $xy\not\in \mathrm{3B}$ by Proposition~\ref{prop:Co2no3B}. Moreover, the $G$-conjugacy class $\mathrm{4C}$ powers to $\mathrm{2B}$~\cite{atlas}, and thus $xy\not\in \mathrm{4C}$. Consequently $xy\in\mathrm{2C}$, and so every pair of $\mathrm{2A}$ elements in $H$ commute. Next, we compute using \magma~\cite{magma} that there is no $z\in \mathrm{2A}\cap C_G(x,y)$ such that $\{xz,yz\}\subseteq \mathrm{2C}$ and hence $H$ has at most two elements of class $2\mathrm{A}$.

    Now, $H$ necessarily normalises $K:=\langle x,y\rangle$. We compute in \magma~the normaliser $N_G(K)$, and generate a list of all subgroups $S$ with $K\leq S\leq N_G(K)$ which contain no $\mathrm{2B}$ or $\mathrm{3B}$ using our program {\tt FilteredSubgroups()}. There are a total of six (conjugacy classes of) such subgroups, the largest of which have order 40. We test these directly with {\tt{NearTITest()}} finding that none yields a binary action. Consequently, $|H\cap \mathrm{2A}|\leq 1$.

    Finally, if there is some $ z\in\mathrm{2C}\cap H$, then by Lemma~\ref{lem:Finf}, $H$ contains an orbit of $C_G(z)$ on the class $\mathrm{2A}$. Since $C_G(z)$ has no central $\mathrm{2A}$-involution (as can be verified in \magma), each such orbit has length greater than one, contradicting that $|H\cap \mathrm{2A}|\leq 1$. Therefore, there is no such $z$, as desired.
    \end{proof}

\begin{prop}\label{prop:Co2bigorderreduction}
    If $2\mid|H|$ then $|H|\mid 2^{4}\cdot3$.
\end{prop}
\begin{proof}
    Assuming $2\mid |H|$, Corollary~\ref{cor:Co2cent2A} implies that $H$ has a central involution of class $\mathrm{2A}$. We deduce that $H<M$, where $M$ is the centraliser of a $\mathrm{2A}$ involution. Thus, $|H|\mid 2^{18}\cdot 3^2\cdot5\cdot7$ by Corollary~\ref{cor:Co2cycsyl}. A Sylow 5-subgroup of $M$ is generated by an element of the $G$-conjugacy class $\mathrm{5B}$, whence by Lemma~\ref{lem:conjconn}, if there is some $g\in H\cap \mathrm{5B}$ then $\Delta(g)\leq H$. But $\Delta(g)=\mathrm{HS}$ by~\cite[Proposition 3.3]{CDP}, contradicting that $|H|\mid 2^{18}\cdot 3^2\cdot5\cdot7$. Therefore, $H$ has no element of order $5$.

    Next, we eliminate the possibility of an element of order 7. We again use {\tt{FilteredSubgroups()}}, to generate a list of all subgroups of $M$ of order divisible by 14 which intersect $\mathrm{2B}\cup\mathrm{2C}\cup\mathrm{3B}$ trivially. We are left with exactly three (conjugacy classes of) such subgroups, of orders 14, 28, and 56. By running {\tt{NearTITest()}} on each we find that none yields a binary action, whence $7\nmid |H|$.

    Now, we generate a list of all subgroups of $M$ of order $32$, finding that each such subgroup has an element of $G$-conjugacy class one of $\mathrm{2B},$ or $\mathrm{2C},$ and so $|H|\mid 2^4\cdot 3^2$. It remains now only to exclude the possibility $3^2\mid |H|$.   By Corollary~\ref{cor:Co2cycsyl}, a Sylow 3-subgroup of $H$ is cyclic, so we use {\tt{FilteredSubgroups()}} to find all subgroups of $M$ of order divisible by 18 but not by 27 which have an element of order 9. We are left with one group of order 18, and one of order 72. Running {\tt{NearTITest()}} on both groups completes the proof.
\end{proof}
\begin{prop}\label{prop:co2oddred}
    Either $H$ is cyclic generated by an element of class $\mathrm{2A}$, or $|H|$ is odd.
\end{prop}
\begin{proof}
    Suppose that $|H|$ is even. By Corollary~\ref{cor:Co2cent2A} and Proposition~\ref{prop:Co2bigorderreduction}, $H<M$ where $M$ is the centraliser of a $\mathrm{2A}$-involution, and $|H|\mid 2^4\cdot 3$. We generate the list of all subgroups of $M$ of order $n\in\{4,8,16,6,12,24,48\}$ using the built-in \magma~\cite{magma} function {\nolinebreak\tt{Subgroups( M : OrderEqual := n )}}, filtering out those with an element of $\mathrm{2B}\cup\mathrm{2C}\cup\mathrm{3B}$. Finally, we run {\tt{NearTITest()}} on each of the remaining groups, finding that none yield binary actions, hence the result.
\end{proof}

\begin{prop}\label{prop:co2no1123}
    If $|H|$ is odd, then $|H|\mid 3^6\cdot5^3\cdot 7$
\end{prop}
\begin{proof}
Suppose that $23\mid |H|$. Examining the maximal subgroups of $G$, as determined in~\cite{W83}, we see that only $\mathrm{M}_{23}$ has order divisible by 23, whence $H\leq \mathrm{M}_{23}$. Moreover, by Lemma~\ref{lem:sub}, the action of $\mathrm{M}_{23}$ on $(\mathrm{M}_{23}:H)$ is binary, a contradiction to Proposition~\ref{prop:1stclas}. 

Suppose now that $11\mid |H|$. From the list of maximal subgroups we deduce that $H$ is contained in one of $\mathrm{U}_6(2)\cn2$, $2^{10}\cn \mathrm{M}_{22}\cn 2$, $\mathrm{McL}$, $\mathrm{HS}\cn 2$, or $\mathrm{M}_{23}$. In fact, this can be refined: since $|H|$ is odd, $H$ is contained in some $S\in\{\mathrm{U}_6(2), \mathrm{M}_{22}, \mathrm{McL}, \mathrm{HS},\mathrm{M}_{23}\}$ (using Lemma~\ref{lem:SZappl} for $\mathrm{M}_{22}$). By Lemma~\ref{lem:sub}, the action of $S$ on $(S:H)$ is binary, and so $S=\mathrm{U}_{6}(2)$ by Propositions~\ref{prop:1stclas} and~\ref{prop:McLclass}. We run {\tt{FilteredSubgroups()}} to generate the list of all subgroups of $\mathrm{U}_6(2)$ which have order divisible by 11, finding that there are exactly two such groups of odd order: one of order 11, and one of order 55. Running {\tt{NearTITest()}}, we see that neither subgroup yields a binary action of $S$, a contradiction, hence the result.
\end{proof}

\begin{prop}\label{prop:co2class}
    Either $H=1$ or $H$ is generated by an element from the $G$-conjugacy class $\mathrm{2A}$.
\end{prop}
\begin{proof}
    By Propositions~\ref{prop:co2oddred} and~\ref{prop:co2no1123} we may assume that $|H|\mid 3^6\cdot 5^3\cdot 7$. We first observe that $H$ is not cyclic of prime order. Indeed, we use the \gap~Character Table Library~\cite{CTblLib1.3.11} to compute that $n_G(\mathcal{C})\geq1216$ for each $\mathcal{C}\in\{\mathrm{3A,3B,5A,5B,7A\}}$, so the claim follows from Lemma~\ref{lem:cycbintest}. Therefore, $|H|$ is divisible by one of $9$, $15$, $21$, $25$ or $35$.

    We now run {\tt{FilteredSubgroups()}} to generate the list of all subgroups of $G$ whose order is odd and divisible by one of $9$, $15$, $21$, $25$ or $35$. We are left with thirty-three conjugacy classes of such subgroups---the largest has order $1215$. We use {\tt{NearTITest()}} on each candidate, which shows that none of the thirty-three yield genuine binary actions, hence the result.
\end{proof}
\subsection{$\mathrm{Fi}_{22}$}
This section will involve our most intensive computations. It is recommended that the interested reader follow along with the supplementary document~\cite{comps} for insight in how the described computations are executed. 

We consider the quasisimple groups $G$ with $G/Z(G)=\mathrm{Fi}_{22}$. By Proposition~\ref{prop:binarydiv}, if  $Z(G)=1$ then $H$ has no elements of order 3 or 5. Moreover, the same holds when $|Z(G)|=3$ by Lemma~\ref{lem:conn3cov} and~\cite[Table 8]{CDP}. Finally, we have the same situation when $2\mid |Z(G)|$ by~\cite[Lemma 2.12]{CDP}.

Additionally, we compute using the \gap~Character Table Library~\cite{CTblLib1.3.11} that $n_G(\mathcal{C})>11$ for each class $\mathcal{C}$ of elements of order either $7$, $11$, or $13$. Consequently, since no maximal subgroup of $G$ has order divisible by $11\cdot 13$, we deduce the following, using Lemma~\ref{lem:cycbintest}.
\begin{prop}
    Either $H$ is a $2$-group or $|H|$ is divisible by at least one of $14$, $22$, $26$, $77$ or $91$.
\end{prop}

\begin{cor}\label{cor:orderformFi22}
    If $H\ne 1$ then $|H|=2^m\cdot p$ for some $p\in\{1,7\}$ and $m\geq 1$.\end{cor}
\begin{proof}
   We run {\tt{FilteredSubgroups()}} on $G$ to generate the list of all subgroups whose order is divisible by an element of $\{22,26,77,91\}$, but indivisible by 3 and 5. We run {\tt{NearTITest()}} on each of the remaining groups to deduce the result.
\end{proof}
By~\cite[Propositions 3.1, 3.2, 3.3, and Theorem 4.2]{CDP} and Lemma~\ref{lem:conjconn}, any element $x\in H$ of maximal 2-fixity satisfies $\varphi(x)\in\mathrm{2A}$. Notably, $\mathrm{2A}$ is a class of $3$-transpositions~\cite{atlas}, that is, any pair of non-commuting $\mathrm{2A}$-involutions has product of order 3---in fact the product of any two commuting $\mathrm{2A}$-involutions is in class $\mathrm{2B}$. In particular, we deduce the following.
\begin{prop}\label{prop:Fi22exist2B}
    If $|\varphi(H)\cap\mathrm{2A}|\geq 2$, then $|\varphi(H)\cap\mathrm{2B}|\geq1$.
\end{prop}
We can also fairly easily show a strong converse.
\begin{prop}\label{prop:fi222B2C}
    For each $g\in H\cap\varphi^{-1}(\mathrm{2B})$, there exist $x,y\in H\cap\varphi^{-1}(\mathrm{2A})$ such that $xy=g$. Moreover, for each $g\in H\cap\varphi^{-1}(\mathrm{2C})$, there exist $x,y,z\in H\cap\varphi^{-1}(\mathrm{2A})$ such that $xyz=g$.
\end{prop}
\begin{proof}
    Let $\mathcal{C}\in\{\mathrm{2B},\mathrm{2C}\}$, and suppose there is some $g\in H\cap \varphi^{-1}(\mathcal{C})$. By Lemma~\ref{lem:conjconn}, and~\cite[Theorems 1.1 and 4.2]{CDP}, there exists some $h\in\varphi^{-1}(\mathrm{2A})$ such that $\fix(g)\subseteq\fix(h)$. Thus, $h^{C_G(g)}\leq H$ by Lemma~\ref{lem:Finf}. We compute in \magma~all orbits of $C_G(g)$ on each conjugacy class in $\varphi^{-1}(\mathrm{2A})$, finding that each orbit either generates a group of order divisible by 3, generates a group intersecting the center non-trivially, or contains elements $x,y,z$ satisfying the claimed conditions. Since $3\nmid |H|$ we deduce the existence result.\end{proof}
\begin{cor}\label{cor:fi22unique2a}
    If $|\varphi(H) \cap \mathrm{2A}|=1$, then $|H|=2$.
\end{cor}
\begin{proof}
            By Proposition~\ref{prop:fi222B2C}, $H$ has a unique involution $h\in\varphi^{-1}(\mathrm{2A})$. In fact, no $G/Z(G)$-conjugacy class of elements of order 4 powers into $\mathrm{2A}$~\cite{atlas}, so the Sylow 2-subgroups of $H$ have order 2, thus $|H|\mid14$ by Corollary~\ref{cor:orderformFi22}. Since $h$ is the unique involution in $H$, we deduce that $H\leq C_G(h)=(Z(G)\times2).\mathrm{U}_6(2)$~\cite{atlas}, and so the action of $C_G(h)$ on $(C_G(h) : H)$ is binary, and moreover, this is equivalent to a binary action of $K:=Z(G).\mathrm{U}_6(2)$ with stabiliser of order either 1 or 7. We compute that $n_K(\mathcal{C})>5$ for each $K$-conjugacy class of elements of order 7 thus $|H|=2$ by Lemma~\ref{lem:cycbintest}.
    \end{proof}

\begin{prop}\label{prop:fi222C}
    If $|\varphi(H) \cap \mathrm{2A}|>1$, then  $|\varphi(H) \cap \mathrm{2A}|\geq 3$ and moreover, for any $x,y,z\in H \cap \varphi^{-1}(\mathrm{2A})$ distinct, the product $xyz$ satisfies $\varphi(xyz)\in \mathrm{2C}$.
\end{prop}
\begin{proof}
    Suppose for contradiction that $x,y\in H$ are the unique $\mathrm{2A}$-lifts in $H$. Then $\varphi(xy)\in\mathrm{2B}$, and by Proposition~\ref{prop:fi222B2C}, $xy$ is the unique $\mathrm{2B}$-lift in $H$ and moreover, $x$, $y$, and $xy$ are the only three involutions in $H$. Additionally, $n_{G/Z(G)}(\mathrm{2A,2A,2B})=2$, and so $H\leq C_G(xy)=N_G(\langle x,y\rangle)$. Thus, if $t\in H$ is such that $t^2=xy$, then either $t\in C_G(\langle x,y\rangle)$, or $y^t=x$. In the latter case, $$(xt)^2=xtxt=xx^{t^{-1}}t^2=(xy)(xy)=1,$$ so $xt\in H\setminus\{x,y,xy\}$ is an involution, contradicting that $H$ has exactly three involutions, whence $t$ centralises $\langle x,y\rangle$. We compute in \magma~that no such element can exist (see~\cite{comps} for full details), and thus since no class of elements of order 4 powers into $\mathrm{2A}$, we deduce that $H=\langle x,y\rangle$ ($C_G(xy)$ has order indivisible by 7~\cite{atlas}). Running {\tt NearTITest()} on $H$ we arrive at a contradiction, thus $|\varphi(H) \cap \mathrm{2A}|\geq 3$.

    For the final claim it suffices to work in $G/Z(G)$. Using \magma~we can easily compute the conjugacy class of $xyz$ for all (conjugacy representatives of) triples $(x,y,z)\in\mathrm{2A}^3$ with $x$, $y$, and $z$ distinct and mutually commuting, finding that $xyz\in\mathrm{2C}$ for all such $(x,y,z)$ hence the result.  
\end{proof}
To finish the proof we shall consider separately the cases $2\mid |Z(G)|$ and $2\nmid |Z(G)|$. We start with the latter case. Since $2\nmid |Z(G)|$, for ease of notation we may identify the involution classes of $3.\mathrm{Fi}_{22}$ with those of $\mathrm{Fi}_{22}$. This family requires the most in depth computation---we run through all computations in the supplementary file~\cite{comps}.

\begin{prop}\label{prop:Fi222local}
    If $|H\cap\mathrm{2C}|\geq 1$, then either $2^{5}\trianglelefteq H\leq 2^{5+8}\cn( S_3\times (Z(G).A_6))$ where all involutions in $H$ are contained in the normal subgroup $2^5$, or $2^{10}\trianglelefteq H\leq 2^{10}\cn(Z(G).\mathrm{M}_{22})$.
\end{prop}
\begin{proof}
    Suppose $g\in H\cap \mathrm{2C}$. Since $\mathrm{2A}$ is the unique class $\mathcal{C}$ of $G$ of maximal 2-fixity, it follows from Lemma~\ref{lem:Finf} that $H$ contains an orbit of $C_G(g)$ on $\mathrm{2A}$. We compute these orbits in \magma~\cite{magma}, finding exactly one orbit which generates a $\{2,7\}$-group---the group $L$ generated is elementary abelian $2^5$ with normaliser $Z(G).(2^{5+8}\cn( S_3\times  A_6))<_{\max}G$, and contains exactly $6$ elements of class $\mathrm{2A}$. 
    
    Let $K:=\langle H\cap\mathrm{2A}\rangle$. Then $K$ is an elementary abelian normal subgroup of $H$. If $|H\cap \mathrm{2A}|=6$, then $K=L$, and we are done. Suppose that $|H\cap\mathrm{2A}|\geq 7$. Then $\langle L,x\rangle\leq K$ for some $x\in\mathrm{2A}\setminus L$. We use \magma~to compute $\langle L,x\rangle$ for each $x\in\mathrm{2A}\setminus L$, finding that it either has order divisible by 3, or is elementary abelian $2^6$---the former is not possible, so we suppose the latter holds. 
    
    By Lemma~\ref{lem:Finf} and the beginning of the proof, for each element $h\in H\cap\mathrm{2C}$, there is a unique $C_G(h)$-orbit $\mathcal{O}(h)$ on $\mathrm{2A}$ such that $\mathcal{O}(h)\subseteq H$. We compute in \magma~that $\langle L,\mathcal{O}(h) : h \in H\cap\mathrm{2C}\rangle$ contains an elementary abelian group of order $2^{10}$ with normaliser $2^{10}\cn(Z(G).\mathrm{M}_{22})$. We compute that if any additional element of $\mathrm{2A}$ is added to our generating set then we generate a group of order divisible by 3, hence the desired result. 
            \end{proof}
\begin{cor}\label{cor:Fi22no2C}
    If $2\nmid |Z(G)|$ then $H\cap \mathrm{2C}=\emptyset$.
\end{cor}
\begin{proof}
    We assume the contrary, so that either $K_1:=2^{10}\trianglelefteq H\leq 2^{10}\cn(Z(G).\mathrm{M}_{22})$ or $K_2:=2^{5}\trianglelefteq H\leq 2^{5+8}\cn( S_3\times (Z(G). A_6))$ by Proposition~\ref{prop:Fi222local}. Suppose first that we are in the former case.

    By Lemma~\ref{lem:sub}, the action of $M_1:=2^{10}\cn(Z(G).\mathrm{M}_{22})$ on $(M_1:H)$ is binary. Since $K_1\trianglelefteq H$, this action is equivalent to a faithful action of $Z(G).\mathrm{M}_{22}$ with point stabiliser isomorphic to $H/K_1$. By Proposition~\ref{prop:1stclas}, $H/K_1=1$ and so $H=K_1$. Next, we run {\tt{NearTITest()}} on $K_1$ in $G$, finding that the action is not binary, therefore, $K_2\trianglelefteq H\leq2^{5+8}\cn( S_3\times  (Z(G).A_6))=:M_2$.

    Since $|M_2|$ is indivisible by $7$, we deduce from~\ref{cor:orderformFi22} that $H$ is a 2-group. Every involution in $H$ is contained in $K_2$, so we compute all 2-subgroups of $G$, all of whose involutions lie in $K_2$ (for details see~\cite{comps}). Testing each such subgroup using {\tt NearTITest()}, shows that none yield genuine binary actions, contradicting our assumption. Therefore, $H\cap \mathrm{2C}=\emptyset$ as desired.\end{proof}
\begin{prop}\label{prop:2fi22no2C}
    Suppose that $2\mid |Z(G)|$. Then $H\cap \varphi^{-1}(\mathrm{2C})=\emptyset$.  
\end{prop}
\begin{proof}
    Assume the contrary and let $g\in H\cap\varphi^{-1}(\mathrm{2C})$. We compute in \magma~that there are exactly four orbits of $C_G(g)$ on $\varphi^{-1}(\mathrm{2A})$ which generate a group of order coprime to 3, each of size exactly three. By Lemma~\ref{lem:Finf}, $H$ contains one of these orbits. In fact, we verify that two orbits $O$ satisfy $g\in\langle O\rangle$, while the other two have $Z(G)\cap \langle g,O\rangle\ne 1$, and so we may assume that $g\in \langle O\rangle$. Moreover, both such orbits generate an elementary abelian group of order $2^3$ containing exactly three elements of $\varphi^{-1}(\mathrm{2A})$ Let $K_1(g)$ and $K_2(g)$ be these groups and observe that the pair $\{K_1(g),K_2(g)\}$ is uniquely defined by $g$, and at least one of the two is a subgroup of $H$ by Lemma~\ref{lem:Finf}. Suppose for the remainder that $K_1(g)\leq H$; the other case is identical.

    Suppose that there is some $h\in (H\cap\varphi^{-1}(\mathrm{2A}))\setminus K_1(g)$. We can compute in \magma~conjugacy representatives of each group of the form $\langle h,K_1(g)\rangle$, and then find the smallest groups $L$ intersecting $Z(G)$ trivially and containing $\langle h,K_1(g)\rangle$ such that for all $g'\in \varphi^{-1}(\mathrm{2C})\cap L$ there is some $i\in\{1,2\}$ such that $K_i(g')\leq L$. We find that there are exactly three such (conjugacy representative) groups $L_1,L_2,L_3$, elementary abelian of orders 16, 32, and 64, respectively. Thus, we can assume that either $|H\cap\varphi^{-1}(\mathrm{2A})|=3$, or one of $L_1,L_2,L_3$ is a subgroup of $H$ by Lemmas~\ref{lem:conjconn} and Lemma~\ref{lem:Finf}.

    Suppose that $\langle L_i,h'\rangle\leq H$ for some $h'\in (H\cap\varphi^{-1}(\mathrm{2A}))\setminus L_i$. We can again compute the smallest group, $\overline{L_i}$ containing $\langle L_i,h'\rangle$ such that $K_j(g')\leq \overline{L_i}$ for all $g'\in \varphi^{-1}(\mathrm{2C})\cap L$, finding that $\overline{L_i}$ is either conjugate to one of $L_1,L_2,L_3$, or intersects $Z(G)$ non-trivially. Thus, $\langle H\cap \varphi^{-1}(\mathrm{2A})\rangle\in\{K_1(g),L_1,L_2,L_3\}$. In fact, all involutions of $H$ lie in this subgroup $\langle H\cap \varphi^{-1}(\mathrm{2A})\rangle$. Indeed, by Proposition~\ref{prop:fi222B2C}, for any 2B- or 2C-lift $k\in H$, there exists either a pair or a triple of 2A-lifts whose product is $k$.

    Now, $H\leq N_G(J)$ for some $J\in\{K_1(g),L_1,L_2,L_3\}$.     Suppose first that $H$ is a 2-group. As established, all involutions in $H$ lie in $J$; we compute in \magma~the conjugacy class representatives of all groups with $|J|-1$ involutions which are generated by $J$ together with additional 2-elements (see~\cite{comps} for details). This list then consists of the conjugacy representatives of all eligible 2-groups $H$; we test the groups with {\tt NearTITest()}, confirming that none yield binary actions. 
    
    On the other hand, if $H$ is not a 2-group, then $7\mid |H|$ by Corollary~\ref{cor:orderformFi22}, and so $7\mid |N_G(J)|$---we verify in \magma~that this is only true for $J=L_3$. We use {\tt FilteredSubgroups()} to generate the list of all subgroups of $N_G(L_3)$ with order divisible by 7 which contain $L_3$; testing each such group with {\tt NearTITest()} we confirm the result.
             \end{proof}
\begin{prop}\label{prop:clasFi22}
    Either $H=1$ or $H$ is generated by an element of $\varphi^{-1}(\mathrm{2A})$.
\end{prop}
\begin{proof}
    Suppose that $H\ne 1$. First observe that $|H|$ is even by Corollary~\ref{cor:orderformFi22}. Thus, Proposition~\ref{prop:fi222C} together with Corollary~\ref{cor:Fi22no2C} and Proposition~\ref{prop:2fi22no2C} imply that $|\varphi(H)\cap\mathrm{2A}|=1$. The result now follows from Corollary~\ref{cor:fi22unique2a}. 
   \end{proof}

\subsection{$\mathrm{HN}$} We move on to $G=\mathrm{HN}$; by~\cite[Propositions 3.1, 3.2, and 3.3]{CDP}, $\Gamma(\mathcal{C})$ is connected for all $\mathcal{C}\in\{\mathrm{2A,2B,3A,3B,5A,5B,5C,5D}\}$. Consequently we have the following.

\begin{prop}\label{prop:firstredHN}
    The stabiliser $H$ has order dividing $7\cdot 11\cdot 19$.
\end{prop}
\begin{proof}
    That $2\nmid|H|$ and $3\nmid |H|$ is Proposition~\ref{prop:binarydiv}; we show that $H$ has no elements of order 5. Suppose the contrary. Since $\Gamma(\mathcal{C})$ is connected for all $\mathcal{C}\in\{\mathrm{5A,5B,5C,5D}\}$, it follows from Lemma~\ref{lem:conjconn} that $\mathrm{5E}$ is the unique $G$-conjugacy class of maximal 5-fixity and so $\Delta(h)\leq H$ for some $h\in\mathrm{5E}$. By~\cite[Proposition 3.3]{CDP}, $\Delta(h)=5^{1+4}_+$; examining this subgroup $5^{1+4}_+$ closer, we see that its center is generated by an element of $G$-conjugacy class $\mathrm{5B}$. Consequently, $\mathrm{5B}$ is a class of maximal $5$-fixity by Lemma~\ref{lem:maxfixdelta}, a contradiction.
\end{proof}

The remaining primes are easy to deal with.

\begin{prop}
    The group $\mathrm{HN}$ has no non-trivial faithful transitive binary actions.
\end{prop}
\begin{proof}
    We show that $H$ has no elements of order $7$, $11$, or $19$. Note that no maximal subgroup of $G$ has order divisible by $7\cdot11\cdot19$. Additionally, $G$ has no elements of order $7\cdot 11$, $7\cdot19$, or $11\cdot 19$, whence $H$ is either trivial or cyclic of order $7$, $11$, or $19$ (here we are using that $p\nmid q-1$ for all primes $p,q\in\{7,11,19\}$). The class multiplication coefficient $n_G(\mathcal{C})$ is exceedingly large for each $\mathcal{C}\in\{\mathrm{7A,11A,19A,19B}\}$ so $|H|$ is not prime by Lemma~\ref{lem:cycbintest}. The result follows.
\end{proof}

\subsection{$\mathrm{Ly}$} The next group to consider is $G=\mathrm{Ly}$. By~\cite[Propositions 3.1, 3.2, and 3.3]{CDP}, $\Gamma(\mathcal{C})$ is connected for all $\mathcal{C}\in\{\mathrm{2A},\mathrm{3B},\mathrm{5A}\}$. We quickly deduce the following.
\begin{prop}
    The stabiliser $H$ has order dividing $3^7\cdot 7\cdot 11\cdot 31\cdot37\cdot 67$
\end{prop}

\begin{proof}
    That $2\nmid |H|$ is Proposition~\ref{prop:binarydiv}. Suppose that $5\mid |H|$. Since $\Gamma(\mathrm{5A})$ is connected it follows that $\mathrm{5B}$ is the unique $G$-conjugacy class of maximal 5-fixity, and so $\Delta(h)\leq H$ for some $h\in\mathrm{5B}$. By~\cite[Proposition 3.3]{CDP}, $\Delta(h)=5^{1+4}_+$ which is verified to intersect $\mathrm{5A}$ non-trivially, and so $\mathrm{5A}$ is a class of maximal $5$-fixity by Lemma~\ref{lem:maxfixdelta}, a contradiction.
\end{proof}
We now proceed in a slightly different manner than for many of the other groups. In particular, for our methods it is preferable to work inside permutation representations of our groups; unfortunately, we only have access to high-dimensional matrix representations of $\mathrm{Ly}$ in \magma~\cite{magma}, and so we instead start by sequentially considering the maximal subgroups of $G$ which have permutation representations available.
\begin{prop}\label{prop:Lymaxes1}
    If $H\ne 1$ then either $H$ is generated by an element of the $G$-conjugacy class $\mathrm{3A}$, or $H$ is not a subgroup of either of the maximal subgroups $G_2(5)$ and $2.A_{11}$.
\end{prop}
\begin{proof}
    We use {\tt{FilteredSubgroups()}} to generate all subgroups of $K\in\{G_2(5), 2.A_{11}\}$ with order divisible by at least one of $3,$ $7$, $11$ or $31$ and indivisible by both 2 and 5, apart from the subgroup generated by an element of $G$-conjugacy class $\mathrm{3A}$. We confirm that none of these subgroups yield binary actions of $K$ by running {\tt{NearTITest()}} on each; the result now follows from Lemma~\ref{lem:sub}.
\end{proof}

\begin{prop}\label{prop:Lymaxes2}
If $1\ne H\leq 3.\mathrm{McL}\cn2$, then either $|H|=9$, or $H$ is generated by an element of $G$-conjugacy class $\mathrm{3A}$.
\end{prop}
\begin{proof}
        Since $|H|$ is odd, $ H\leq K:=3.\mathrm{McL}$, and the action of $K$ on $(K:H)$ is binary by Lemma~\ref{lem:sub}. Therefore, by Proposition~\ref{prop:McLclass}, $H$ is conjugate to a subgroup of $\langle Z(K),t\rangle$ where $t$ is a lift of the $\mathrm{McL}$ class 3A. Using the \gap~ Character Table Library~\cite{CTblLib1.3.11}, we check class fusion, verifying that $t$ fuses to the $G$-conjugacy class 3A, as do the non-trivial central elements, hence the result. 
                 \end{proof}
Examining class fusion using the Character Table Library, we see that the maximal $5^3.\mathrm{L}_{3}(5)$ has no class fusing to $\mathrm{3A}$ (which is the unique class $\mathcal{C}$ of 3-elements such that $\Gamma(\mathcal{C})$ is disconnected~\cite[Proposition 3.1]{CDP}). Consequently, if $H\leq5^3.\mathrm{L}_{3}(5)$ yields a binary action of $G$, then $3\nmid |H|$. This leaves us with the possibility that such $H$ is either trivial or cyclic of order 31. Since $n_G(\mathcal{C})>29$ for every class of 31-elements we deduce the following result.

\begin{prop}\label{prop:Lymax3}
    If $H\ne 1$ then $H$ is not a subgroup of the maximal subgroup $5^3.\mathrm{L}_3(5)<_{\max}G.$
\end{prop}

The next result allows us to eliminate the presence of elements of order $37$ and $67$, leaving us with one maximal subgroup to consider.

\begin{prop}
    Neither $37$ nor $67$ divide the order of the stabiliser $H$. In particular, $H\leq3^5\cn(2\times\mathrm{M}_{11})$
\end{prop}
\begin{proof}
    Suppose first that $37\mid |H|$. Examining the list of maximal subgroups~\cite{atlas}, we deduce that $H\leq37\cn18$. Now, there is no element of order 18 which powers into the $G$-conjugacy class $\mathrm{3A}$, whence $H$ is cyclic of order 37. But $n_G(\mathrm{37A})>35$, contradicting Lemma~\ref{lem:cycbintest}.

    Next we consider the case that $67\mid |H|$. In this case $H\leq 67\cn11$, and moreover, $H$ is not cyclic as $n_G(\mathrm{67A})=n_G(\mathrm{67B})=n_G(\mathrm{67C})>65$. Therefore, we may write $H=\langle x\rangle\cn\langle y\rangle$ for some $x$ of order 67 and $y$ of order $11$.  Since $|C_G(y)|=66$~\cite{atlas}, we may take $t\in C_G(y)$ of order $3$. Since $3\nmid |N_G(\langle x\rangle)|$ it follows that $H\cap H^t=\langle y\rangle$, and so $H$ has an orbit on which it is Frobenius, contradicting Lemmas~\ref{lem:frob} and~\ref{lem:suborb}.

    To deduce the final claim we begin by noting that by Lemmas~\ref{prop:Lymaxes1}, ~\ref{prop:Lymaxes2}, and~\ref{prop:Lymax3}, $|H|$ is indivisible by $7$ and $31$, and so $|H|\mid 3^7\cdot 11$, and moreover, $11\mid |H|$ only if $H\leq 3^5\cn(2\times\mathrm{M}_{11})$. On the other hand, if $11\nmid |H|$, then $H$ is a 3-group---since $3^5\cn(2\times\mathrm{M}_{11})$ contains a Sylow 3-subgroup of $G$ we deduce that $H\leq 3^5\cn(2\times\mathrm{M}_{11})$. 
\end{proof}
\begin{prop}
    Either $H=1$ or $H$ is generated by an element of the $G$-conjugacy class $\mathrm{3A}$.
\end{prop}
\begin{proof}
Suppose that $H\ne 1$. We construct $K:=3^{5}\cn(2\times\mathrm{M}_{11})$ in \magma~\cite{magma} as a subgroup of $G$ using a straight-line program available from~\cite{webatlas}, and use the \magma~command {\tt{PermutationRepresentation()}} to build a permutation representation of $K$. We then use our program {\tt{FilteredSubgroups()}} to generate the list of all subgroups of $K$ of order divisible by at least one of 3 or 11, and indivisible by any other prime. We next run {\tt{NearTITest()}} in $K$ on each of the remaining groups; the only groups for which it does not succeed in discounting have orders $\{3,3^3,3^4,3^5,3^6,3^7,3^5\cdot11\}.$ There is a unique remaining class of subgroups of order $3^5\cdot 11$ and the action of $3^5\cn\mathrm{M}_{11}$ on the cosets of each of these subgroups is equivalent to the action of the $\mathrm{M}_{11}$ with stabiliser of order 11. This action is non-binary by Proposition~\ref{prop:1stclas}, and so the action of $G$ with stabiliser $3^5\cn 11$ is non-binary by Lemma~\ref{lem:sub}. Therefore, $|H|\in\{3,3^3,3^4,3^5,3^6,3^7\}$.

Since $H$ is a $3$-group, we can find $H$ in $3.\mathrm{McL}\cn 2$ ($3.\mathrm{McL}\cn2$ contains a Sylow 3-subgroup of $G$). Therefore, $|H|\in\{3,9\}$ by Proposition~\ref{prop:Lymaxes2}. But $|H|\ne 9$ by the previous paragraph, whence $H$ is generated by an element of the $G$-conjugacy class $\mathrm{3A}$, as was to be shown.
\end{proof}
\subsection{$\mathrm{Th}$} The challenging work for $G=\mathrm{Th}$ was done in~\cite{CDP}. Indeed, by Proposition~\ref{prop:binarydiv}, if the action of $G$ on $(G:H)$ is binary then $|H| \mid 13\cdot19\cdot31$ so we need only consider the possibility of the primes $13$, $19$, $31$.
\begin{prop}
   The group $\mathrm{Th}$ has no non-trivial faithful transitive binary actions.
\end{prop}

\begin{proof}
    We first observe that $p\nmid q-1$ for all $p,q\in\{13,19,31\}$, whence $G$ has no non-cyclic subgroup of order a product of more than one element in $\{13,19,31\}$. Moreover, $G$ has no cyclic subgroups of order larger than 39, and so if $H\ne 1$ then it is cyclic of prime order. Computing $n_G(\mathcal{C})$ for each $\mathcal{C}\in\{\mathrm{13A,19A,31A,31B}\}$ we find the result is much larger than $p-2$ so $H=1$ by Lemma~\ref{lem:cycbintest}, as was to be shown.
\end{proof}
\subsection{$\mathrm{Fi}_{23}$}
We next consider $G=\mathrm{Fi}_{23}$. Using the results of~\cite{CDP} we can quickly reduce ourselves to the case that $H$ is a 2-group.

\begin{prop}\label{prop:Fi232group}
    The stabiliser $H$ is a 2-group.
\end{prop}

\begin{proof}
    By Proposition~\ref{prop:binarydiv}, $|H|\mid 2^{18}\cdot7\cdot11\cdot13\cdot17\cdot23$. We run {\tt{FilteredSubgroups()}} to generate the list of all subgroups of $G$ whose order is divisible by at least one of $7$, 11, 13, 17, or 23, and indivisible by 3 and 5. We next verify that none of the remaining subgroups yield genuine binary actions by running {\tt{NearTITest()}} on each candidate, and thus deduce that $H$ is a 2-group.
 \end{proof}
By~\cite[Propositions 3.1, 3.2, and 3.3]{CDP}, $\Gamma(\mathrm{2B})$ and $\Gamma(\mathrm{2C})$ are connected, whence $\mathrm{2A}$ is the unique class of maximal $p$-fixity. Thus, provided that $H\ne 1$, then $|H\cap\mathrm{2A}|\geq 1$; the next proposition gives us an upper bound.
   \begin{prop}\label{prop:fi23clas}
    Either $H=1$ or $H$ is generated by an element from the $G$-conjugacy class $\mathrm{2A}.$
\end{prop}
\begin{proof}
    Since $H$ is a 2-group, we can find $H$ in a maximal subgroup $M=2.\mathrm{Fi}_{22}$. Therefore, $H\leq \langle Z(M),t\rangle$ where $t$ is a lift of the $M/Z(M)$-class 2A by Proposition~\ref{prop:clasFi22}. Since the $G$-class 2A is the unique class of maximal 2-fixity, either $H=1$, $H$ is generated by an element of 2A, or $H=\langle Z(M),t\rangle$. We generate the group $\langle Z(M),t\rangle$ in \magma~as a subgroup of $G$ and test it with {\tt NearTITest()}, confirming that it does not yield a binary action, hence the result.
    \end{proof}

\subsection{$\mathrm{Co}_{1}$} Consider now the case $G/Z(G)=\mathrm{Co}_1$. By Proposition~\ref{prop:binarydiv} and~\cite[Theorem 4.2]{CDP}, if the action of $G$ on $(G:H)$ is binary then $|H|\mid 3^9\cdot11\cdot13\cdot23$. We can quickly reduce this to the situation where $H$ is a 3-group.
\begin{prop}
    The stabiliser $H$ is a 3-group.
\end{prop}

Recall that we are assuming that $H\cap Z(G)=1$.

\begin{proof}
    We run {\tt{FilteredSubgroups()}} to generate the list of all subgroups of $G$ with order divisible by at least one of $11$, $13$, or $23$ but indivisible by 2, 5, and 7. Finally, we run {\tt{NearTITest()}} on each of the remaining groups to deduce the result.
\end{proof}

For ease of notation we identify the classes of elements of order 3 of $2.\mathrm{Co}_1$ with their images in the simple quotient. By~\cite[Propositions 3.1, 3.2, and 3.3]{CDP} together with~\cite[Lemma 2.12]{CDP}, the $G$-conjugacy class $\mathrm{3A}$ is the unique class $\mathcal{C}$ of elements of order 3 such that $\Gamma(\mathcal{C})$ is disconnected, and thus is the unique class of maximal 3-fixity by Lemma~\ref{lem:conjconn}. We shall show over the next two propositions that $H$ contains no elements of $\mathrm{3B}\cup\mathrm{3C}$.

\begin{prop}\label{prop:co1no3c}
    The stabiliser $H$ contains no elements of the $G$-conjugacy class $\mathrm{3C}$.
\end{prop}
\begin{proof}
    Suppose, for contradiction, that there exists $g\in H\cap\mathrm{3C}$. Since the normaliser $N_G(\langle g\rangle )=Z(G).(3^{1+4}.2\mathrm{U}_4(2).2)$ is maximal~\cite{atlas}, Lemma~\ref{lem:normalsub} implies that $H$ has a subgroup $K$ which is normal in $N_G(\langle g\rangle)$, and also contains an element of the $G$-conjugacy class $\mathrm{3A}$. Checking class fusion using the {\sf GAP} Character Table Library~\cite{CTblLib1.3.11}, we see that $O_3(N_G(\langle g\rangle))=3^{1+4}$ contains no element of $\mathrm{3A}$, and so $K$ must have order not a power of 3, a contradiction.
\end{proof}
\begin{prop}\label{prop:co1no3B}
    The stabiliser $H$ contains no element of the $G$-conjugacy class $\mathrm{3B}$. 
\end{prop}
\begin{proof}
    Suppose, for contradiction, that there exists $g\in H\cap\mathrm{3B}$. For ease of notation we assume $Z(G)=1$; the other case is similar. Now, $N_G(\langle g\rangle )=3^{2}.\mathrm{U}_4(3).2^2$ can only be a subgroup of the maximals $3^{2}.\mathrm{U}_4(3).D_8$ and $3.\mathrm{Suz}\cn2$, and is maximal in both~\cite{atlas}. Therefore, $H$ has a subgroup $K$ which is normal in either $N_G(\langle g\rangle)$, $3^{2}.\mathrm{U}_4(3).D_8$, or $3.\mathrm{Suz}\cn2$ and properly contains a subgroup generated by an element of the $G$-conjugacy class $\mathrm{3A}$ by Lemma~\ref{lem:normalsub}. Clearly, the only possible such subgroup is $O_3(N_G(\langle g\rangle)) =3^2=K\leq H$. 
    
    Suppose now that $H$ contains an element $x\in\mathrm{3A}\setminus K$---we shall examine the structure of all possible such groups $\langle K,x\rangle$ to arrive at a contradiction. It suffices to consider representatives $x$ of each $N_G(K)$-orbit on $\mathrm{3A}$---we generate a list of these using our \magma~function {\tt OrbReps()}, and find that each possible group of the form $\langle K,x\rangle$ either has order divisible by 2 or has an element $y\in \mathrm{3C}$, a contradiction. Therefore, $\mathrm{3A}\cap H\leq K$. Since the group $K$ was uniquely defined by $g$ and $K$ is generated by its elements in 3A, this implies that $\mathrm{3B}\cap H\leq K$; elements of $\mathrm{3A}\cup\mathrm{3B}$ are not the power-up of any 3-elements in $G$~\cite{atlas}, whence $H=K$. Running {\tt NearTITest()} on $K$ we reach a contradiction, thus $H\cap\mathrm{3B}=\emptyset$.
\end{proof}
This is enough of a reduction to classify the binary actions of $G$.
\begin{prop}
    Either $H=1$, or $Z(G)\ne 1$ and $H$ is generated by an element of the $G$-conjugacy class $\mathrm{3A}$.
\end{prop}
\begin{proof}
    Suppose $g,h\in\mathrm{3A}\cap H$; we first show that $h\in\langle g\rangle$. We compute using the {\sf GAP} Character Table Library~\cite{CTblLib1.3.11} that $$n_G(\mathrm{3A},\mathrm{3A},\mathcal{C})=\begin{cases}
        2&\text{if $\C=\mathrm{3B}$}
        \\0&\text{if $\C\not\in\{\mathrm{1A},\mathrm{3A},\mathrm{3B}\}$ is a class of $3$-elements.}
    \end{cases}$$
    Consequently, since $H\cap\mathrm{3B}=\emptyset$, we deduce that all non-trivial elements of $\langle g,h\rangle$ are in $\mathrm{3A}$. Now, $\Gamma(\mathrm{3A})$ is edgeless by \cite[Proposition 3.1]{CDP}, and so $G$ has no elementary abelian subgroup of order 9 with all eight non-trivial elements in $\mathrm{3A}$, and so $h\in \langle g\rangle$, as was to be shown, thus $H\leq N_G(\langle g\rangle)=(Z(G)\times3).\mathrm{Suz}\cn2$~\cite{atlas}. Since $|H|$ is odd, $H\leq 3.\mathrm{Suz}$, and so $H\leq \langle g\rangle$ by Proposition~\ref{prop:suzclass}. Finally, $$n_G(\mathrm{3A})=\begin{cases}
        1&\text{ if $Z(G)\ne 1$}\\
        5347&\text{ if $Z(G)=1$},
    \end{cases}$$
    so the result follows from Lemma~\ref{lem:cycbintest}.
    \end{proof}
\subsection{$\mathrm{J}_4$} We next consider the largest of Janko's simple groups, $G=\mathrm{J}_4$. By Proposition~\ref{prop:binarydiv}, $|H|\mid 5\cdot7\cdot11^3\cdot23\cdot29\cdot31\cdot37\cdot43.$ The group $G$ has 13 conjugacy classes of maximal subgroups, as determined by Kleidman and Wilson~\cite{KW88}; we write representatives in order of increasing index as $K_1,K_2,\dots,K_{13}$.

\begin{prop}\label{prop:J4bigorderreduction}
The stabiliser $H$ has order dividing $5\cdot7\cdot 11^3$.
\end{prop}
\begin{proof}
    Suppose first that $p\mid |H|$ for some $p\in\{43,37,29\}$, and take $x\in H$ with order $p$. If $p\in\{43,29\}$, the only maximal subgroups of $G$ containing $x$ are $N_G(\langle x\rangle)<p\cn(p-1)$.  If $p=37$, then either $H\leq N_G(\langle x\rangle)=K_{13}=37\cn12$ or else $H\leq K_5= \mathrm{U}_3(11).2$.
    
    Since neither $2$ nor $3$ divide $|H|$, it follows in all cases that either $H=\langle x\rangle$, or $H=\langle x\rangle\cn 7$ (we are using here that $\mathrm{U}_3(11)$ has a unique maximal subgroup of order divisible by 37, which has structure $37\cn 3$). We compute that $n_G(\mathcal{C})>p-2$ for each class of elements of order $p$, and thus $H=\langle x\rangle\cn \langle y\rangle$ for some $y$ of order 7 by Lemma~\ref{lem:cycbintest} (and in particular we may already deduce that $p\ne37)$. Next $|C_G(y)|=840$~\cite{atlas}, and thus $y$ is centralised by some element $t$ of order $5$---no such element normalises $\langle x\rangle$, so $H\cap H^t=\langle y\rangle$, and thus $H$ has an orbit on which it is Frobenius, and hence not binary by Lemma~\ref{lem:frob}. Thus, $G$ is non-binary by Lemma~\ref{lem:suborb}, a contradiction.

    Suppose now that $23\mid |H|$. Then by Lemma~\ref{lem:SZappl} either $H\leq \mathrm{M}_{24}<2^{11}\cn\mathrm{M}_{24}=K_1$, or $H\leq \mathrm{L}_2(23)\cn2=K_9$. The former contradicts Proposition~\ref{prop:1stclas} by Lemma~\ref{lem:sub}, and the latter---together with our divisibility conditions---implies that $H\leq23\cn11$. But then we again have that $H\leq \mathrm{M}_{24}$, a contradiction, whence $23\nmid |H|.$

    Finally, if $31\mid |H|$, then using Lemma~\ref{lem:SZappl} we deduce that either $H\leq \mathrm{L}_{5}(2)<2^{10}\cn\mathrm{L}_5(2)=K_3$ or $H\leq\mathrm{L}_{2}(32)\cn5=K_8$. Considering the subgroups of both $\mathrm{L}_5(2)$ and $\mathrm{L}_{2}(32)$ of order indivisible by 2 and 3, we deduce that $H\leq 31\cn5=:\langle x\rangle\cn\langle y\rangle$. Now, $H$ cannot be cyclic by Lemma~\ref{lem:cycbintest} as $n_G(\mathcal{C})>29$ for each class $\mathcal{C}$ of elements of order 31, thus $H=\langle x\rangle\cn\langle y\rangle$. But $7\mid|C_G(y)|$ and $7\nmid |N_G(\langle x\rangle)|$~\cite{atlas}, and so there is some $t\in G$ such that $H\cap H^t=\langle y\rangle$. This contradicts Lemmas~\ref{lem:frob} and~\ref{lem:suborb}, hence the result.
\end{proof}
We next consider the situation where $11^{2}\mid |H|$. 
\begin{prop}
    The stabiliser $H$ has order dividing $ 5\cdot7\cdot 11$.
\end{prop}
\begin{proof}
    Suppose that $11^2\mid |H|$. Examining the orders of the maximal subgroups of $G$, we deduce that either $H\leq\mathrm{U}_{3}(11)<\mathrm{U}_3(11)\cn 2=K_5$, or $H\leq 11^{1+2}\cn(5\times2S_4)=K_7$. In each case the only subgroups of order indivisible by 2 and 3 are contained in $11^{1+2}\cn5$ which can be found in $\mathrm{U}_3(11)$ as a subgroup of a Sylow normaliser. We generate the list of all subgroups of $\mathrm{U}_{3}(11)$ of order divisible by $11^2$, and test each with {\tt{NearTITest()}}, deducing that $H=11^{1+2}\cn5$ (this group escapes detection by {\tt{NearTITest()}} since the transitive action of $\mathrm{U}_{3}(11)$ with stabiliser $11^{1+2}\cn5$ has base size 3). It is not hard to find $t\in\mathrm{U}_{3}(11)$ such that $|H\cap H^t|=5$. Thus, since the elements of $G$ of order $11$ have centralisers of order coprime to 5, $G$ has a suborbit on which it acts as a Frobenius group, a contradiction to Lemmas~\ref{lem:frob} and~\ref{lem:suborb}. 
\end{proof}
\begin{prop}
    The group $\mathrm{J}_4$ has no non-trivial faithful transitive binary actions.
\end{prop}
\begin{proof}
    Any group of order $5\cdot7\cdot11$ has an element of order 77, so cannot be a subgroup of $G$. Thus, since $G$ has no elements of order 55, the only remaining possible groups are cyclic of orders 5, 7, 11, or 35, or the group $11\cn5$. The cyclic groups of prime order can be eliminated as the relevant class multiplication coefficients are large. Any cyclic group of order 35 may be found in the maximal subgroup $K_4$ which we construct in \magma~\cite{magma} using a straight-line program from ~\cite{webatlas}; running {\tt{NearTITest()}} on such a subgroup in $K_4$ reveals that this action of $K_4$ is not binary, whence neither is the corresponding action of $G$ by Lemma~\ref{lem:sub}. We can find both classes of $11\cn5$ subgroups in $\mathrm{U}_{3}(11)$---running {\tt{NearTITest()}} on each reveals that neither gives a genuine binary action of $G$, by Lemma~\ref{lem:sub}.
\end{proof}
\subsection{$\mathrm{Fi}_{24}'$} When $G/Z(G)=\mathrm{Fi}_{24}'$ we need not get our hands dirty. By Proposition~\ref{prop:binarydiv} and---in the case that $|Z(G)|=3$---Lemma~\ref{lem:conn3cov},~\cite[Lemma 2.12]{CDP} and the data of~\cite[Table 8]{CDP}, if the action of $G$ on $(G:H)$ is binary, then $|H|\mid 11\cdot13\cdot17\cdot23\cdot29$. From here it is enough to look at class multiplication coefficients and the orders of maximal subgroups to deduce that $H=1$.
\begin{prop}
    The group $G$ has no non-trivial faithful transitive binary actions.
\end{prop}
\begin{proof}
    Let $S:=\{11,13,17,23,29\}$. If $H<\mathrm{Fi}_{23}< G$, then by Lemma~\ref{lem:sub}, the action of $\mathrm{Fi}_{23}$ on $\mathrm{Fi}_{23}/H$ is binary, and so $|H|\leq 2$ by Proposition~\ref{prop:fi23clas}, a contradiction. Consequently, after examining the orders of all maximal subgroups of $G$~\cite{LW91}, we deduce that $|H|$ is divisible by at most two primes from $S$.

    Suppose that $|H|$ is divisible by exactly two primes from $S$. Again, by examining the order of maximals we deduce that $|H|\in\{11\cdot13,11\cdot17,11\cdot23\}$. Since $G$ has no elements of order $11\cdot13$, $11\cdot17$, or $11\cdot23$, we deduce that $H=23\cn11=:\langle x\rangle\cn\langle y\rangle$. Since $3\mid C_G(y)$ but $3\nmid N_G(\langle x\rangle)$~\cite{atlas} we reach a contradiction to Lemmas~\ref{lem:frob} and~\ref{lem:suborb}, so either $H=1$ or $|H|\in S$. We compute that $n_G(\mathcal{C})>p-2$ for each $G$-conjugacy class of elements of order $p\in  S$, hence the result follows from Lemma~\ref{lem:cycbintest}.
\end{proof}
\subsection{$\mathbb{B}$} Suppose $G/Z(G)=\mathbb{B}$. The group $\mathbb{B}$ has four classes of involutions; by~\cite[Propositions 3.1, 3.2, 3.3]{CDP} the class 2A is the only involution class with disconnected graph. Similarly by~\cite[Theorem 4.2]{CDP} the class 2b is the unique non-central involution class of $2.\mathbb{B}$ with disconnected graph. By Proposition~\ref{prop:binarydiv}, the stabiliser of any transitive binary action of $G$ has order dividing $2^{41}\cdot11\cdot13\cdot17\cdot19\cdot23\cdot31\cdot47$. Let $\varphi:G\to G/Z(G)$ be the quotient map.

\begin{prop}\label{prop:Bno2B}
    The group $\varphi(H)$ contains no elements of the $G/Z(G)$-conjugacy class $\mathrm{2B}.$
\end{prop}

\begin{proof}
    Suppose for contradiction that there is an involution $g\in H\cap\varphi^{-1}(\mathrm{2B})$. Suppose first that $|Z(G)|=1$.  Since $C_G(g)$ is maximal and $\mathrm{2A}$ is the unique $G$-conjugacy class of involutions $\mathcal{C}$ for which $\Gamma(\mathcal{C})$ is disconnected~\cite[Propositions 3.1, 3.2, and 3.3]{CDP}, we deduce from Lemma~\ref{lem:normalsub} that $H$ contains some $K\trianglelefteq C_G(g)=2^{1+22}.\mathrm{Co}_2$ properly containing $\langle g\rangle$. There is a unique such proper subgroup, namely $K=O_2(C_G(g))=2^{1+22}$. Now, checking class fusion from $C_G(g)$ to $G$ using the \gap~Character Library~\cite{CTblLib1.3.11}, we see that $g$ is in the $C_G(g)$-class $\mathrm{2a}$, and that $C_G(g)$ has exactly two classes fusing to the $G$-class $\mathrm{2A}$: the class labelled $\mathrm{2f}$ lies outside of $K$ and fuses to $\mathrm{2A}$, and the class $\mathrm{2b}$ lies inside $K$ and fuses to $\mathrm{2A}$. Moreover, $n_{C_G(g)}(\mathrm{2f},\mathrm{2f},\mathrm{2a})=n_{C_G(g)}(\mathrm{2f},\mathrm{2b},\mathrm{2a})=0,$ and so $K$---and hence $H$---contains every pair of $\mathrm{2A}$ elements of $G$ whose product is $g$ (using that an element in such a pair necessarily commutes with $g$).

    Let $x\in\mathrm{2f}$. We compute that $n_{C_G(g)}(\mathrm{2f},\mathrm{2c},\mathrm{2f})=63,$ so for each $x\in\mathrm{2f}$ we may find some $y\in\mathrm{2f}$ such that $xy\in\mathrm{2c}$. The class $\mathrm{2c}$ lies in $K$ and fuses to $\mathrm{2B}$ in $G$, so we now repeat the argument of the preceding paragraph with $xy$ in place of $g$ to deduce that $x\in H$. Since $x$ was arbitrary, we deduce that $\langle O_2(C_G(g)),\mathrm{2f}\rangle\leq H$, but $\langle O_2(C_G(g)),\mathrm{2f}\rangle\trianglelefteq C_G(g)$, and so $H=C_G(g)$ contradicting Cherlin's Conjecture~\cite{GLS22}.
    
    Finally, suppose that $|Z(G)|=2$. Again, since $C_G(g)$ is maximal, and $\mathrm{2b}$ is the unique $G$-conjugacy class of involutions $\mathcal{C}$ for which $\Gamma(\mathcal{C})$ is disconnected, we deduce from Lemma~\ref{lem:normalsub} that $H$ contains some $K\trianglelefteq C_G(g)=2.(2^{1+22}.\mathrm{Co}_2)$ properly containing $\langle g\rangle$ but intersecting $Z(G)$ trivially. Such a subgroup satisfies $\varphi(K)\in\{2^{1+22},2^{1+22}.\mathrm{Co}_2\}$. Since $K\cap Z(G)=1$, $\varphi|_K$ is injective whence $K\in\{2^{1+22},2^{1+22}.\mathrm{Co}_2\}$.
    Next, we can check from the character tables of $G$ and $2^{1+22}.\mathrm{Co}_2$ that $G$ can have no subgroup of the form $2^{1+22}.\mathrm{Co}_2$ (so the extension $2.(2^{1+22}.\mathrm{Co}_2)$ does not split). Consequently, $K=2^{1+22}$ whence $C_G(g)/K$ is a central extension $2.\mathrm{Co}_2$. But $\mathrm{Co}_2$ has trivial Schur multiplier; taking the preimage of a complement $\mathrm{Co}_2$ we get a subgroup $2^{1+22}.\mathrm{Co}_2< G$, a contradiction, hence the result.
\end{proof}
The $\mathbb{B}$-class $\mathrm{2C}$ lifts to elements of order 4 in $2.\mathbb{B}$, but when $Z(G)=1$ we cannot immediately rule out the presence of a $\mathrm{2C}$ element in $H$. However, if $H$ contains such an element, $g$, say, then since $C_G(g)$ has no central $\mathrm{2A}$ element we deduce from Lemmas~\ref{lem:conjconn} and~\ref{lem:Finf} that $|H\cap\mathrm{2A}|\geq 2$. In fact, we will never get more than two elements from $\mathrm{2A}$.
\begin{prop}\label{prop:B2C}
    If $Z(G)=1$ then $H$ contains at most two elements of the $G$-conjugacy class $\mathrm{2A}$ with equality if and only if $|H\cap\mathrm{2C}|=1.$ If $Z(G)=2$ then $H$ has at most one involution $g$ with $\varphi(g)\in\mathrm{2A}$.
\end{prop}
\begin{proof}
    The product of any two distinct $\mathrm{2A}$ involutions in $G/Z(G)$ lies in one of $\mathrm{2B},\mathrm{2C},\mathrm{3A}$, or $\mathrm{4B}$~\cite{atlas}---since $H$ contains no elements of order 3 or lifts of the $G/Z(G)$-conjugacy class $\mathrm{2B}$, we deduce that for any $x,y\in H\cap \varphi^{-1}(\mathrm{2A})$ distinct, $xy\in\varphi^{-1}(\mathrm{2C})$ (using that $\mathrm{4B}$ powers up to $\mathrm{2B}$ in $G$~\cite{atlas}). Since $\mathrm{2C}$ lifts to elements of order 4 in $2.\mathbb{B}$ we deduce the claimed result when $Z(G)=2$. For the remainder we shall assume $Z(G)=1$.

    Suppose we have $x,y,z\in H\cap \mathrm{2A}$ distinct. Then $xy,xz,yz\in\mathrm{2C}$, and so $x,y$, and $z$ centralise  $xy\in\mathrm{2C}$. Using the \gap~Character Table Library~\cite{CTblLib1.3.11}, we compute all possible class fusions from $C_G(xy)$ to $G$ finding that $x$ and $y$ are in the $C_G(xy)$-class labelled $\mathrm{2h}$, $xy\in\mathrm{2d}$ and either $z\in\mathrm{2i}$ or $z\in\mathrm{2j}$---the coming calculations give the same results for both possibilities of $z$, so without loss of generality assume that $z\in\mathrm{2i}$. We compute that $$n_{C_G(xy)}(\mathrm{2h},\mathrm{2i},\mathcal{C})=\begin{cases}
        1&\text{if $\mathcal{C}\in\{\mathrm{2a,2e}\}$,}\\
        0&\text{otherwise.} 
    \end{cases}$$
    Since $\mathrm{2a}$ fuses to $\mathrm{2B}$ in $G$, it follows from the above and Proposition~\ref{prop:Bno2B} that $yz\in\mathrm{2e}$. Finally, for involution classes $\mathcal{C}$ of $C_G(g)$ we compute that $n_{C_G(xy)}(\mathrm{2d},\mathrm{2e},\mathcal{C})\ne0$ if and only if $\mathcal{C}=\mathrm{2a}$, and so since $xy\in\mathrm{2d}$ and  $yz\in \mathrm{2e}$ we deduce that $xz=xy\cdot yz\in\mathrm{2a}$ which fuses to $\mathrm{2B}$ in $G$, a contradiction, hence $|H\cap 2\mathrm{A}|\leq 2$.

    Suppose, finally, that $|H\cap \mathrm{2A}|=2$; we shall show that $|H\cap \mathrm{2C}|=1$. The centraliser in $\mathbb{B}$ of a 2C-element $g\in H$ is $(2^2\times F_4(2))\cn2$. This group $C_G(g)$ centralises no $\mathrm{2A}$-element thus $x\ne y\in H\cap\mathrm{2A}$ form an orbit of $C_G(g)$ on $\mathrm{2A}$ by Lemma~\ref{lem:Finf}. But then $xy\in\mathrm{2C}$ by the first paragraph, so $xy=yx$ is centralised by $C_G(g)$. Therefore, since $Z(C_G(g))=\langle g\rangle$ we deduce that $xy=g$ and $\langle x,y\rangle$ is a normal subgroup $2^2$ in $C_G(g)$. Thus the pair $x,y$ is uniquely defined by $g$; the result now follows.    
   \end{proof}
 In particular, Propositions~\ref{prop:Bno2B} and~\ref{prop:B2C} tell us that one of the following holds: \begin{itemize}
    \item $|H|$ is odd;
    \item $H$ contains a unique element of $\varphi^{-1}(\mathrm{2A})$ and so $H\leq Z(G).(2.{}^2E_6(2).2)$; or
    \item $Z(G)=1$ and $H$ contains a unique $\mathrm{2C}$ involution, and so $H\leq(2^2\times F_4(2))\cn2$.
\end{itemize}

To eliminate the final case we begin by excluding the presence of $\mathrm{2D}$ elements.

\begin{prop}\label{prop:Bno2D}
    The stabiliser $H$ contains no element of the $G$-conjugacy class $\varphi^{-1}(\mathrm{2D})$.
\end{prop}
\begin{proof}

   We begin by showing that the centraliser of an involution from $\varphi^{-1}(\mathrm{2D})$ does not centralise an involution of $\varphi^{-1}(\mathrm{2A})$ nor an involution of 2C when $Z(G)=1$. For ease of notation we assume $Z(G)=1$, the argument is identical otherwise (ignoring 2C lifts). Let $x\in\mathrm{2D}$, let $y\in\mathrm{2A}\cup\mathrm{2C}$, and suppose for contradiction that $C_G(x)$ centralises $y$. Then $x\in C_G(y)$, and $C_G(y,x)=C_G(x)$, in particular, $C_G(y)$ has a class fusing to $\mathrm{2D}$ whose centralisers in $C_G(y)$ have order $|C_G(x)|$. We check using the \gap~Character Table Library~\cite{CTblLib1.3.11} that this is not the case, and so the claim follows. 
    We continue now with no assumption on $Z(G)$. Suppose for contradiction that there is some $g\in\varphi^{-1}(\mathrm{2D})\cap H$. Then there exists an $h\in\varphi^{-1}(\mathrm{2A})$ satisfying $h^{C_G(g)}\subseteq H$ by Lemmas~\ref{lem:conjconn} and Lemma~\ref{lem:Finf}. By the above, $|h^{C_G(g)}|\geq 2$---this is a contradiction when $|Z(G)|=2$ by Proposition~\ref{prop:B2C}, so for the remainder $Z(G)=1$. By Proposition~\ref{prop:B2C}, $h^{C_G(g)}=\{h,h'\}$ for some $h'\in\mathrm{2A}$. Consequently, $C_G(g)$ centralises $hh'\in\mathrm{2C}$. This is a contradiction to the first paragraph, hence the result.
\end{proof}

\begin{prop}
    The stabiliser $H$ contains at most one element of the $G$-conjugacy class $\varphi^{-1}(\mathrm{2A})$, and no elements of  $\varphi^{-1}(\mathrm{2C})$.
\end{prop}
\begin{proof}
    When $|Z(G)|=2$ this is just Proposition~\ref{prop:B2C}, so we assume $Z(G)=1$.
    
    Suppose otherwise, so that $|H\cap\mathrm{2A}|=2$ and $|H\cap\mathrm{2C}|=1$ by Proposition~\ref{prop:B2C}. Then $N:=2^2\trianglelefteq H\leq(2^2\times F_4(2))\cn2=:K$. Moreover, by Propositions~\ref{prop:Bno2B},~\ref{prop:B2C}, and~\ref{prop:Bno2D}, all involutions in $H$ lie in the normal subgroup $N$. As a result, the action of $K$ on $(K:H)$ is equivalent to that of $\overline{K}=F_4(2)\cn2$ on the cosets of a subgroup $\overline{H}=H/N$ with order indivisible by 3, 5, and 7 by Proposition~\ref{prop:binarydiv}. 

Now, the classes of involutions in the subgroup $F_4(2)\trianglelefteq \overline{K}$ all have connected graph by~\cite[Theorem 2]{GGL25}. Additionally, using the \gap~Character Table Library~\cite{CTblLib1.3.11} we compute that the $\overline{K}$-conjugacy class $\mathcal{C}$ of involutions not in the normal $F_4(2)$ satisfies $n_{\overline{K}}(\mathcal{C})=45744$; since this conjugacy class has maximal centraliser~\cite{atlas} we deduce that $\Gamma(\mathcal{C})$ is connected as well by~\cite[Lemma 2.14]{CDP}. Therefore, $|\overline{H}|$ is odd by Lemma~\ref{lem:conjconn}, whence $|\overline{H}|\mid 13\cdot17$.

    Since $F_4(2)\cn2$ has no subgroups of order $13\cdot 17$, we deduce that $\overline{H}$ is cyclic of order either $1$, $13$, or $17$. We compute that $n_{\overline{K}}(\mathcal{C})>p-2$ for all classes of elements of order $p\in\{13,17\}$, and so $\overline{H}=1$ by Lemma~\ref{lem:cycbintest}. Therefore, $H=N$.

    To complete the proof we must show that the action of $G$ on $(G:N)$ is not binary; we do this by appealing to Lemma~\ref{lem:efficienttest}. Write $N\cap\mathrm{2A}=\{g,h\}$. Since $n_G(\mathrm{2A},\mathrm{2A},\mathrm{2C})=2$, for each $z\in\mathrm{2C}$ there is a unique pair $x_z,y_z\in\mathrm{2A}$ such that $x_zy_z=z$. We now use a counting argument---the aim is to show that their exists a pair $t,r\in\mathrm{2C}$ such that $tr(gh)=1$ and $g,h\not\in\{x_t,y_t,x_r,y_r\}$, from which the result will follow by Lemma~\ref{lem:efficienttest} using $H=\langle g,h\rangle$, $H_2=\langle x_t,y_t\rangle$ and $H_3=\langle x_r,y_r\rangle$. Set $A=\{(t,r)\in\mathrm{2C}\times\mathrm{2C} : trgh=1\}$. Then $|A|=n_G(\mathrm{2C})=184246272$. Now, set $A_g=\{t\in \mathrm{2C} : g\in\{x_t,y_t\}\}$. Observe that $t\in A_g$ if and only if either $tx_t=g$ or $ty_t=g$, whence $|A_g|=n_G(\mathrm{2C,2A,2A})=23113728$. Next, $$|\{(t,r)\in A \mid \{t,r\}\cap A_g\ne\emptyset\}|\leq 2|A_g|,$$ since the pair $(t,r)\in A$ is uniquely defined by either entry. Defining $A_h$ similarly, we deduce that there are at most $$|\{(t,r)\in A \mid \{t,r\}\cap A_g\ne\emptyset\}\cup\{(t,r)\in A \mid \{t,r\}\cap A_h\ne\emptyset\}|\leq 4|A_g|,$$ pairs $(t,r)\in A$ such that $\{g,h\}\cap \{x_t,y_t,x_r,y_r\}\ne\emptyset$.

Therefore, the number of pairs $(t,r)$ satisfying $trgh=1$ and $g,h\not\in\{x_t,y_t,x_r,y_r\}$ is at least $$184246272
-4\cdot23113728=91791360$$ hence the result.\end{proof}

\begin{prop}\label{prop:Bclas}
    Either $H=1$ or $H$ is generated by an involution from the $G$-conjugacy class $\varphi^{-1}(2\mathrm{A})$.
\end{prop}
\begin{proof}
    Suppose throughout that $H\ne 1$. By Propositions~\ref{prop:Bno2B},~\ref{prop:Bno2D}, and ~\ref{prop:B2C}, either $|H|\mid 11\cdot13\cdot17\cdot19\cdot23\cdot31\cdot47$, or $H$ has a unique involution, and this involution is of $G$-conjugacy class $\varphi^{-1}(\mathrm{2A})$.

    Suppose first that $|H|\mid 11\cdot13\cdot17\cdot19\cdot23\cdot31\cdot47$. Since $G$ has no elements of order at least $11\cdot 13$, either $|H|$ is prime, or $H\in\{23\cn11,47\cn23\}$ (note $G$ has a unique class of both subgroups~\cite{W99}). We compute that $n_G(\mathcal{C})>p-2$ for each class of elements of relevant prime order $p$, and so $H\in\{23\cn11,47\cn23\}$. 
    
    Suppose $H=47\cn23$ and write $H=\langle x\rangle\cn\langle y\rangle$ where $x$ has order $47$ and $y$ has order 23. Then $11\nmid |N_G(\langle x\rangle)|$ so there is some $t\in G$ of order $11$ such that $\langle y\rangle^t=\langle y\rangle$ and $\langle x\rangle^t\cap\langle x\rangle=1.$ In particular, $H\cap H^t=\langle y\rangle$, a contradiction to Lemma~\ref{lem:frob} and Lemma~\ref{lem:suborb}. On the other hand, the subgroup $23\cn11$ can be found in $\mathrm{Fi}_{23}$, which has no non-trivial binary actions with odd stabiliser by Proposition~\ref{prop:fi23clas}, hence the action of $G$ with stabiliser $23\cn11$ is not binary by Lemma~\ref{lem:sub}. Therefore, $|H|$ is even.

    Let $x$ be the unique involution in $H$. Then $H\leq C_G(x)=Z(G).(2.{}^2E_6(2).2)$---since the $G$-conjugacy class $\varphi^{-1}(\mathrm{2A})$ is not the power up of any class of $2$-power order, we deduce that the Sylow 2-subgroup of $H$ is generated by $x$. Consequently, the action of $C_G(x)$ on $C_G(x)/H$ is equivalent to an action of $C_G(x)/\langle x\rangle$ with stabiliser of order dividing $11\cdot 13\cdot17\cdot19$, and moreover this action is binary by Lemma~\ref{lem:sub}. Since $p\nmid q-1$ for each pair $p,q\in\{11,13,17,19\}$ (and ${}^2E_6(2)$ has no elements of order at least $143$), we deduce that either $|H/\langle x\rangle|$ is prime, or $H=\langle x\rangle$. We compute that $n_{C_G(x)/\langle x\rangle}(\mathcal{C})>p-2$ for each ${}^2E_6(2).2$-class of $p$-elements, $p\in\{11,13,17,19\}$. Therefore, $H=\langle x\rangle$ by Lemma~\ref{lem:cycbintest}, as desired. 
\end{proof}
\subsection{$\mathbb{M}$} The final group we must consider is $G=\mathbb{M}$. Despite recent advances~\cite{Sey24}, many computations in the monster group remain a challenge. Luckily, the monster is very well-behaved when it comes to this problem, so we need not rely heavily on computations (indeed this is also the case for~\cite{CDP}). By Proposition~\ref{prop:binarydiv}, $|H| \mid 17\cdot19\cdot23\cdot29\cdot31\cdot41\cdot47\cdot59\cdot71$. Throughout this section we define $S=\{17,19,23,29,31,41,47,59,71\}$.
\begin{prop}\label{prop:Mnocyc}
    The stabiliser $H$ is not cyclic of prime order. 
\end{prop}
\begin{proof}
    We compute the class multiplication coefficients $n_G(\mathcal{C})$ for every class $\C$ of elements of order $p\in S$. In each case $n_G(\mathcal{C})>p-2$, so the result follows from Lemma~\ref{lem:cycbintest}.
\end{proof}
Since $119<17\cdot 19$ is the largest element order in $G$ and $|H|$ is squarefree we deduce from Proposition~\ref{prop:Mnocyc} that $H\in\{59\cn29,47\cn 23\}$. Note that $59\cn 29$ is a maximal subgroup~\cite{DLP25}, and so cannot yield a binary action as this would contradict Cherlin's Conjecture~\cite{GLS22}. Moreover, the subgroup $47\cn 23\leq G$ can be found in $2.\mathbb{B}$ and so cannot yield a binary action by Proposition~\ref{prop:Bclas} and Lemma~\ref{lem:sub}. Thus we have shown the following.
  \begin{prop}\label{prop:Mclas}
    The group $\mathbb{M}$ has no non-trivial faithful transitive binary actions.
\end{prop}
 
Combining the results of this section we have now proved Theorem~\ref{thm:maintrans}.
\section{Intransitive actions}\label{sec:intrans}
The vast majority of research on relational complexity has focused on the transitive case. By \cite[Lemma 2.5]{GG23} each orbit of a binary action induces a transitive binary permutation group, so to classify \emph{all} binary actions we first need to build some machinery to detect which combinations of binary orbits induce non-binary actions. This was initiated in~\cite{GG23}, however the situation for the sporadic groups is more complicated, and thus demands a broader approach.
\subsection{General results}

Let $\Omega$ be a $G$-set and let $I$ be an $n$-tuple of points of $\Omega$. For $m\leq n$ we write $I_{(m)}$ to denote the $(n-1)$-tuple obtained from $I$ by removing the $m$th coordinate.
\begin{lemma}\label{lem:distincttups}
    Let $\Omega$ be a $G$-set. Suppose that $I=(\alpha_1,\dots,\alpha_n)\in \Omega^n$ is such that $G_{\alpha_s}\geq G_{\alpha_t}$ for some $1\leq s\ne t\leq n$, and that $J=(\alpha_1',\dots,\alpha_n')$, is such that there exists some $y\in G$ satisfying $\alpha_s'=\alpha_s^y$ and $\alpha_t'=\alpha_t^y$. Then $I\stb{n} J$ if and only if $I_{(s)}\stb{n-1}J_{(s)}$.
\end{lemma}
\begin{proof}
    The forward implication is clear. Suppose that $I_{(s)}\stb{n-1} J_{(s)}$. Then there is some $x\in G$ such that $I_{(s)}^x=J_{(s)}$, and thus $\alpha_t^x=\alpha_t'=\alpha_t^y$. That is $yx^{-1}\in G_{\alpha_t}\leq G_{\alpha_s}$. But then $\alpha_s^x=(\alpha_s^{yx^{-1}})^x=\alpha_s^y$, and so $I\stb{n} J.$
\end{proof}

For convenience, we shall now restate a few lemmas from~\cite{GG23}.

\begin{lemma}[{\cite[Lemma 2.5]{GG23}}]\label{GGlem2.5}
Let $A$ and $B$ be $G$-sets with $B\subseteq A$. Then $\mathrm{RC}(G,B)\leq\mathrm{RC}(G,A)$.
\end{lemma}

\begin{lemma}[{\cite[Lemma 2.6]{GG23}}]\label{GGlem2.6}
Let $A$, $B$, and $C$ be disjoint $G$-sets. If $B$ and $C$ are isomorphic as $G$-sets, then $\mathrm{RC}(G,A\sqcup B\sqcup C)=\mathrm{RC}(G,A\sqcup B).$
\end{lemma}

\begin{lemma}[{\cite[Lemma 2.7]{GG23}}]\label{GGlem2.7}
Let $G$ act on the disjoint union of $G$-sets $A\sqcup B$ and suppose that the action on $A$ is either trivial or semi-regular. Then $\mathrm{RC}(G,A\sqcup B)=\mathrm{RC}(G, B).$
\end{lemma}

We call a pair $G,H$ of groups \emph{comparable} if either $H\leq G$ or $G\leq H$. Observe that if $H_1$, $H_2$, and $H_3$ are subgroups with some pair of them comparable, then $H_1\cap (H_2\cdot H_3)\subseteq H_1\cap((H_1\cap H_2)\cdot H_3)$. Indeed, this is trivial when either $H_2\leq H_1$, $H_1\leq H_2$, or $H_1\leq H_3$; if $H_3\leq H_1$, and $x\in H_1\cap (H_2\cdot H_3)$ then we can write $x=h_2h_3$ with $h_2\in H_2$ and $h_3\in H_3\leq H_1$. But then $h_2=xh_3^{-1}\in H_1\cap H_2$ and thus $x\in H_1\cap((H_1\cap H_2)\cdot H_3)$. Finally, if $H_3\leq H_2$ then and $ H_1\cap (H_2\cdot H_3)=H_1\cap H_2\subseteq H_1\cap((H_1\cap H_2)\cdot H_3),$ and similarly if $H_2\leq H_3$. These observations will save us some case checking down the road, and are used immediately in the statement of the following key lemma.

\begin{lemma}\label{lem:intransconditions2}
    Let $A$ and $B$ be $G$-sets satisfying the following restrictions:\begin{enumerate}
        \item[1.] $\mathrm{RC}(G,A)=2$ and $\mathrm{RC}(G,B)=2$;
        \item[2.] if $H_1, H_2,$ and $H_3$ are point stabilisers of points in $A$, $A\sqcup B$, and $B$, respectively such that no two $H_i$ are comparable, then $H_1\cap (H_2\cdot H_3)\subseteq H_1\cap((H_1\cap H_2)\cdot H_3)$;
        \item[3.] there exists a subgroup $N$ contained in every $A$-stabiliser such that $H_1\cap H_2=N$ for every pair $H_1,H_2$ of incomparable $A$-stabilisers, and if $L$ is contained in the intersection of two incomparable $B$-stabilisers then $(L\cdot H_1)\cap (L\cdot H_2)\subseteq L\cdot N$; and
        \item[4.] whenever $H$ is an $A$-stabiliser and $K_1,K_2,K_3$ are incomparable $B$-stabilisers         then for any $L_1\leq K_1\cap K_2$ and $L_2\leq K_1\cap K_3$ we have $(L_1\cdot H) \cap (L_2\cdot H)\subseteq (L_1\cap L_2)\cdot H$.
    \end{enumerate}
    Then $\mathrm{RC}(G,A\sqcup B)=2$.
\end{lemma}
\begin{proof}
Suppose $$I=(\alpha_1,\alpha_2,\dots,\alpha_s,\beta_1,\beta_2,\dots,\beta_t)\stb{2} (\alpha_1',\alpha_2',\dots,\alpha_s',\beta_1',\beta_2',\dots,\beta_t')=J.$$
    We may assume without loss of generality that $\alpha_i\in A$ and $\beta_i\in B$ for all $i$. Moreover, since the action of $G$ on $A$ is binary, by replacing $I$ with $I^g$ for an appropriate $g$, we may assume that $\alpha_i=\alpha_i'$ for all $i$. Similarly we may assume that there is some $y\in G$ such that $J=(\alpha_1,\dots,\alpha_s,\beta_1^y,\dots,\beta_t^y)$. Finally, we may assume that no two of the points in $I$ from $A$ have comparable stabilisers and similarly that no two of the points in $I$ from $B$ have comparable stabilisers by Lemma~\ref{lem:distincttups}.
    
    Suppose first that $t>1$. Consider the subtuples $\widetilde{I}:=(\alpha_1,\beta_1,\beta_2)$, and $\widetilde{J}:=(\alpha_1,\beta_1^y,\beta_2^y)$. Since $I\stb{2}J$, we deduce that $\widetilde{I}\stb{2}\widetilde{J}$, and so there exists $g\in H_1:=G_{\alpha_1}$ such that $\beta_1^g=\beta_1^y$. Additionally, $\widetilde{I}\stb{2}\widetilde{J}^{g^{-1}}$, so there is an $h_1\in H_1$ such that $\beta_2^{h_1}=\beta_2^{yg^{-1}}$, and an $h_2\in H_2:=G_{\beta_1}$ such that $\beta_2^{h_2}=\beta_2^{yg^{-1}}$. Therefore,  $h_3:=h_2h_1^{-1}\in H_3:=G_{\beta_2}$, and so $h_1^{-1}=h_2^{-1}h_3\in H_1\cap(H_2\cdot H_3)$. 
        Now, Condition 2 implies that $h_1^{-1}=h_2'h_3'$ for some $h_2'\in H_1\cap H_2$ and $h_3'\in H_3$. Finally, $h_2^{-1}h_3=h_2'h_3'$, so $h_2h_2'\in H_3$, whence $\beta_2^{h_2'^{-1}}=\beta_2^{h_2}=\beta_2^{yg^{-1}}$, and so by setting $x=h_2'^{-1}g$ we have that $\widetilde{I}^{x}=\widetilde{J}$. Thus $yx^{-1}\in G_{\beta_1}\cap G_{\beta_2}$, and so $y\in(G_{\beta_1}\cap G_{\beta_2})\cdot G_{\alpha_1}$. Repeating the above argument with each $\beta_i\ne \beta_1$ in place of $\beta_2$ we deduce that $y\in  (G_{\beta_1}\cap G_{\beta_i})\cdot G_{\alpha_1}$ for each $i$. Thus, by applying Condition 4 inductively, we deduce that $y \in (\bigcap_{i=1}^t G_{\beta_i})\cdot G_{\alpha_1}$. Arguing similarly for each $\alpha_i$, we get $y\in (\bigcap_{i=1}^t G_{\beta_i})\cdot G_{\alpha_i}$ and so
    $$y\in \left(\bigcap_{i=1}^t G_{\beta_i}\right)\cdot \left(\bigcap_{i=1}^s G_{\alpha_i}\right)$$ by Condition 3. Therefore, we can find $a\in \left(\bigcap_{i=1}^s G_{\alpha_i}\right)$ and $b\in (\bigcap_{i=1}^t G_{\beta_i})$,  such that $y=ba$. But then $\beta_i^{a}=\beta_i^{ba}=\beta_i^{y}$ and $\alpha_i^a=\alpha_i$, and so $I^a=J$, as desired.
    
    Now, suppose that $t=1$---we may assume that $s>1$ for otherwise there is nothing to show. Consider the subtuples $\bar{I}=(\alpha_1,\alpha_2,\beta_1)$ and $\bar{J}=(\alpha_1,\alpha_2,\beta_1^y)$. Since $\bar{I}\stb{2}\bar{J}$, we may find $g_1\in G_{\alpha_1}$ and $g_2\in G_{\alpha_2}$ satisfying $\beta_1^{g_i}=\beta_1^y$. Thus $g_2g_1^{-1}\in G_{\beta_1}$, and so $g_1^{-1}\in G_{\alpha_1}\cap(G_{\alpha_2}\cdot G_{\beta_1})$.     Condition 2 now implies that there exist $g_2'\in (G_{\alpha_1}\cap G_{\alpha_2})=N$ and $g_3'\in G_{\beta_1}$ such that $g_1^{-1}=g_2'g_3'$. But then $\beta_1^{g_2'^{-1}}=\beta_1^{g_3'g_1}=\beta_1^y$. But $N\leq G_{\alpha_i}$ for all $i$, thus $I^{g_2'^{-1}}=J$, hence the result follows in this case.
\end{proof}
The conditions of Lemma~\ref{lem:intransconditions2} may appear complicated and contrived, but in practice they are relatively straightforward to check. In many situations we will be able to use the following slightly simpler corollary.

\begin{cor}\label{cor:intransconditions}
    Let $A$ and $B$ be $G$-sets satisfying the following restrictions:\begin{enumerate}
        \item[1.] $\mathrm{RC}(G,A)=2$ and $\mathrm{RC}(G,B)=2$;
        \item[2.] If $H_1, H_2,$ and $H_3$ are point stabilisers of points in $A$, $A\sqcup B$, and $B$, respectively such that no two $H_i$ are comparable, then $H_1\cap (H_2\cdot H_3)\subseteq H_1\cap((H_1\cap H_2)\cdot H_3)$; and
        \item[3.] There exists a subgroup $N$ contained in every $A$-stabiliser such that $H_1\cap H_2=N$ for every pair $H_1,H_2$ of incomparable $A$-stabilisers, and $K_1\cap K_2\leq N$ for every pair of incomparable $B$-stabilisers. 
    \end{enumerate}
    Then $\mathrm{RC}(G,A\sqcup B)=2$.
\end{cor}
\begin{proof} 
We must check that the conditions of Lemma~\ref{lem:intransconditions2} hold; of course, Conditions 1 and 2 are the same, so we need only check \ref{lem:intransconditions2}.3 and \ref{lem:intransconditions2}.4.

Suppose $L$ is contained in the intersection of two incomparable $B$-stabilisers. Then $L\leq N$ by \ref{cor:intransconditions}.3, and so $L\leq H_1\cap H_2$ for any two incomparable $A$-stabilisers. Consequently, $(L\cdot H_1)\cap (L\cdot H_2)=H_1\cap H_2=N=L\cdot N$ and so \ref{lem:intransconditions2}.3 holds.

Finally, suppose that $H$ is an $A$ stabiliser, $K_1,K_2,K_3$ are incomparable $B$-stabilisers, and that $L_1\leq K_1\cap K_2$ and $L_2\leq K_1\cap K_3$. Then $L_1,L_2\leq N\leq H$, and so $$(L_1\cdot H)\cap (L_2\cdot H)=H\cap H=H=(L_1\cap L_2)\cdot H.$$ Therefore, all conditions of Lemma~\ref{lem:intransconditions2} are satisfied, hence the result.
           \end{proof}
\begin{cor}\label{cor:intranscharcalc}
    Suppose that $H$ and $K$ are non-conjugate prime order subgroups of $G$ such that the actions of $G$ on $A=(G:H)$ and $B=(G:K)$ are binary. Suppose that there are $G$-classes $\mathcal{C}_H$ and $\mathcal{C}_K$ such that the non-trivial elements of $H$ (resp. $K$) lie in $\mathcal{C}_H$ (resp. $\C_K$). If $n_G(\mathcal{C}_H,\C_K,\C_H)=n_G(\mathcal{C}_K,\C_K,\C_H)=0$ then the action of $G$ on $A\sqcup B$ is binary.
\end{cor}
\begin{proof}
 Conditions 1 and 3 of Corollary~\ref{cor:intransconditions} clearly hold since all stabilisers have prime order. Let $L_1$, $L_2$, and $L_3$ be distinct stabilisers for $A$, $A\sqcup B$, and $B$, respectively. Suppose $1\ne \ell\in L_1\cap(L_2\cdot L_3)$ so that $\ell=\ell_2\ell_3$ with $\ell\in L_1$, $\ell_2\in L_2$ and $\ell_3\in L_3$. Since $n_G(\mathcal{C}_H,\C_K,\C_H)=n_G(\mathcal{C}_K,\C_K,\C_H)=0$ and $L_1\cap L_3=1$, we deduce that $\ell_3=1$ and $\ell_2=\ell$, but then $L_1=L_2$, a contradiction. Therefore, $L_1\cap(L_2\cdot L_3)=1$ so all conditions of Corollary~\ref{cor:intransconditions} hold, hence the result.
\end{proof}

                  Given a $G$-set, $\Omega$, recall that $\mathcal{S}(\Omega)$ is the set of all conjugacy classes of stabilisers of points of $\Omega$.
\begin{lemma}\label{lem:centralstabs}
    Let $\Omega$ be a $G$-set. Suppose $N\trianglelefteq G$ is a cyclic group. If $\mathcal{S}(\Omega)$ consists of (conjugacy classes of) subgroups only of $N$ then $\mathrm{RC}(G,\Omega)=2$.
\end{lemma}
\begin{proof}
    Since $N$ is cyclic, each of its subgroups are normal in $G$ and so the action on the cosets of $H\leq N$ is binary. Now, the subgroup lattice of a cyclic group is distributive, so given three subgroups $H_1,H_2,H_3\leq N$ we have that \begin{align*}H_1\cap(H_2\cdot H_3)&=(H_1\cap H_2)\cdot(H_1\cap H_3)\\&=((H_1\cap H_2)\cdot H_1)\cap((H_1\cap H_2)\cdot H_3)\\&=H_1\cap ((H_1\cap H_2)\cdot H_3).\end{align*} The distributive properties also show that the inclusions of Conditions 3 and 4 of Lemma~\ref{lem:intransconditions2} always hold for subgroups of $N$. Therefore, arguing inductively (at each stage setting the new $B$ to be the union of the previous $A$ and $B$, and setting the new $A$ to be any new subgroup of $N$), all conditions of Lemma~\ref{lem:intransconditions2} are satisfied, hence the result.
                    \end{proof}
\begin{rk}
For us, Lemma~\ref{lem:centralstabs} will only be used with the stronger assumption that $N$ is a central subgroup of $G$.
\end{rk}
\begin{lemma}\label{lem:cycstabs}
    Suppose that the action of $G$ on $\Omega$ is binary, that $z\in Z(G)$ and that $g\in G\setminus Z(G)$ is of prime order $p$ such that $\langle g\rangle^G\geq 2$ and $g$ is conjugate to all of its non-trivial powers but is not conjugate to $gt$ for any $t\in Z(G)\setminus \{1\}$.         Suppose there exist integers $i,j$ such that the following hold:\begin{multicols}{2}\begin{itemize}\item[(1)]$\langle gz^i\rangle \ne\langle gz^j\rangle$,\item[(2)]$\{\langle gz^i\rangle^G,\langle gz^j\rangle^G\}\subseteq \mathcal{S}(\Omega)$, \item[(3)]$Z(G)\cap\langle gz^i\rangle=1$, \item[(4)]~$z^{i-j}\not\in\langle gz^j\rangle$.\end{itemize}\end{multicols} Then $\mathcal{S}(\Omega)\cap\{C^G,\langle g,C\rangle^G : \langle z^{i-j}\rangle\leq C\leq Z(G)\}=\emptyset.$

       \end{lemma}
\begin{proof}
                   Consider $H_1\in\{C,\langle g,C\rangle\}$ for some $\langle z^{i-j}\rangle\leq C\leq Z(G)$, $H_2=\langle g'z^i\rangle$, and $H_3=\langle g'z^j\rangle$ for some distinct conjugate $g'$ of $g$. Then $H_1\cap H_2=1$ by (3) together with the fact that $g$ is not conjugate to $gt$ for any $t\in Z(G)\setminus \{1\}$. Therefore, $H_1\cap((H_1\cap H_2)\cdot H_3)=H_1\cap H_3\not\ni z^{i-j}$ by (4), while $z^{i-j}\in H_1\cap (H_2\cdot H_3)$, so the result now follows from Lemma~\ref{lem:basiccriteria}.
\end{proof}

\begin{lemma}\label{lem:intrans}
    Let $g\in G\setminus Z(G)$ have prime order $p$, and let $1\ne z\in Z(G)$.     Suppose $H_1\leq \langle g,z\rangle$ and that the action of $G$ on $(G:\langle g,z\rangle)$ is binary.     Suppose that $|\langle g\rangle^G|\geq 3$ and that $g$ is conjugate to each of its non-trivial powers, but is not conjugate to $gt$ for any $t\in Z(G)\setminus \{1\}$. Fix integers $i$ and $j$ and suppose that $H_2\in\{\langle gz^i\rangle,\langle g,z^i\rangle\}^G$ and $H_3\in \{\langle gz^j\rangle,\langle g,z^j\rangle\}^G$ are such that $Z(G)H_1, Z(G)H_2,$ and $Z(G)H_3$ are distinct. Then $H_1\cap (H_2\cdot H_3)\subseteq H_1\cap((H_1\cap H_2)\cdot H_3)$.
\end{lemma}
\begin{proof}
     Since $Z(G)H_1,Z(G) H_2,$ and $Z(G)H_3$ are distinct, we may find conjugates $K_1=\langle g,z\rangle\geq H_1$, $K_2=\langle g_2,z\rangle\geq H_2$, and $K_3=\langle g_3,z\rangle\geq H_3$ of $\langle g,z\rangle$ where no two of the cyclic subgroups $\langle g_i\rangle$ are equal (using that $|\langle g\rangle^G|\geq 3$). Therefore, since the action of $G$ on $\langle g,z\rangle $ is binary, \begin{align*}
            H_1\cap (H_2\cdot H_3)&\subseteq K_1\cap (K_2\cdot K_3)\\&\subseteq K_1\cap ((K_1\cap K_2)\cdot K_3)\\&=K_1\cap (\langle z\rangle\cdot K_3)\\&\subseteq\langle z\rangle.
        \end{align*}
    
     We will assume that $H_2=\langle g_2z^i\rangle$ and $H_3=\langle g_3,z^j\rangle$ for some $g_2\ne g_3\in g^G$; the other cases are very similar. Suppose $1\ne z_0\in H_1\cap (H_2\cdot H_3)\subseteq\langle z\rangle$. Then $z_0=(g_2^mz^{im})(g_3^nz^{jk})$ where $m,n,k$ are integers. Then, $g_2^m\in g_3^{-n}Z(G)$, so since $g$ is not conjugate to $gt$ for any $t\in Z(G)\setminus{1}$ we deduce that $g_2^m=g_3^{-n}$. But $\langle g_2\rangle\cap\langle g_3\rangle=1$, so $p\mid m$, hence $z_0=z^{im+jk}\in \langle z^{\gcd(pi,j)}\rangle\cap H_1$. Thus, if $z^s$ generates $H_1\cap Z(G)$, then $z_0\in \langle z^{\lcm(\gcd(pi,j),s)}\rangle$.
    Finally, \begin{align*}
             H_1\cap ((H_1\cap H_2)\cdot H_3)&= H_1\cap (\langle z^{\lcm(pi,s)}\rangle\cdot H_3)\\&=H_1\cap \langle g_3,z^{\gcd(\lcm(pi,s),j)}\rangle\\&= \langle z^{\lcm(\gcd(\lcm(pi,s),j),s)}\rangle.
        \end{align*}
    Now, $\lcm(\gcd(pi,j),s)=\gcd(\lcm(pi,s),\lcm(j,s))=\lcm(\gcd(\lcm(pi,s),j),s)$, from the distributive properties of $\lcm$ and $\gcd$, hence the result.
   \end{proof}

For the next two results we will only consider centres of prime order $p\in\{2,3\}$.

      Call a subgroup $K\leq G$ \emph{purely binary} if for every $H<G$, the action of $G$ on $(G:H)$ is binary if and only if $H$ is contained in a conjugate of $K$. In general such subgroups of quasisimple groups with $|K|$ not a prime seem very rare, but we will encounter them a couple of times, in the specific scenario of the coming lemmas.

\begin{lemma}\label{lem:p^2intrans2}
    Let $G$ be a group with $Z(G)$ of order $2$. Suppose that $K\geq Z(G)$ is an elementary abelian purely binary subgroup of $G$ of order $4$. Write $K=\langle z,g\rangle$ where $z\in Z(G)\setminus\{1\}$ and $|\langle g\rangle^G|\geq 3$. The action of $G$ on a $G$-set $\Omega$ is binary if and only if $\mathcal{S}(\Omega)\setminus\{\{1\},\{G\}\}$ is contained in at least one of
      $$\{\langle z\rangle^G,\langle g\rangle^G,K^G\}\setminus\{\langle gz\rangle^G\},\,\,\,\,\,\{\langle z\rangle^G,\langle gz\rangle^G,K^G\}\setminus\{\langle g\rangle^G\},\,\,\text{ or }\,\, \{\langle g\rangle^G, \langle gz\rangle^G\},$$    \end{lemma}
\begin{proof}
If $g$ is not conjugate to $gz$ then the forward implication follows immediately from Lemma~\ref{lem:cycstabs}. On the other hand, if $\langle g\rangle^G=\langle gz\rangle^G$ then the statement to prove is that $\mathcal{S}(\Omega)\setminus\{\{1\},\{G\}\}$ is a subset of either $\{\langle z\rangle ^G,K^G\}$ or $\{\langle g\rangle ^G\}$. That is, it suffices to show that the action of $G$ on $(G:\langle g\rangle)\sqcup(G:\langle z,g^i\rangle)$ is not binary for $i\in \{0,1\}$. This is easy: set $H_1=\langle g\rangle$, $H_2=\langle gz\rangle$, and $H_3=\langle z,g'^i\rangle$ where $\langle g'\rangle\not\in\{\langle g\rangle,\langle gz\rangle\}$. Then $H_1\cap(H_2\cdot H_3)=H_1$, but $H_1\cap H_2=1$ so $H_1\cap ((H_1\cap H_2)\cdot H_3)=H_1\cap H_3=\{1\}$ whence the action is not binary by Lemma~\ref{lem:basiccriteria}. This completes the proof of the forward implication.

To prove the converse we may assume throughout that no orbit has stabiliser $1$ or $G$, and that no two orbits are equivalent by~Lemmas~\ref{GGlem2.7} and~\ref{GGlem2.6}, respectively.     
We claim that the conditions of Corollary~\ref{cor:intransconditions} are satisfied when we take $(G:K)$ to be $A$ and $(G:Z(G))$ to be $B$. Indeed, there is a unique $B$-stabiliser, and it is contained in every $A$-stabiliser hence all conditions are satisfied vacuously.     Therefore, the action of $G$ on $A\sqcup B$ is binary by Corollary~\ref{cor:intransconditions}. This completes the proof in the case that $g$ is conjugate to $gz$, so for the remainder of the proof we assume that $\langle g\rangle^G\ne\langle gz\rangle^G$.
    
    Let $h\in \{g,gz\}$ and observe that $h$ is not conjugate to $hz$.  Define $A':=A\sqcup B$ (with $A$ and $B$ as defined in the previous paragraph), and $B':=(G:\langle h\rangle)$. Then Condition 3 of Corollary~\ref{cor:intransconditions} clearly holds with respect to $A'\sqcup B'$, as does Condition 1 by the previous argument. Let $H_1,H_2,$ and $H_3$ be point stabilisers of $G$ on $A'$, $A'\sqcup B'$, and $B'$, respectively, no two of which are comparable. Then $H_2\ne Z(G)$ as this would force $H_2\subseteq H_1$. Moreover, if $H_1=Z(G)$ then $H_2=\langle h_2\rangle$ and $H_3=\langle h_3\rangle$ for conjugates $h_2,h_3$ of $h$, whence $H_1\cap (H_2\cdot H_3)=1$, satisfying Condition 2. Otherwise, $H_1\cap (H_2\cdot H_3)\subseteq H_1\cap((H_1\cap H_2)\cdot H_3)$ by Lemma~\ref{lem:intrans}, and so the action of $G$ on $A'\sqcup B'$ is binary, as desired. 
               
It remains only to show that if we take $A=(G:\langle g\rangle)$ and $B=(G : \langle gz\rangle )$, then the action of $G$ on $A\sqcup B$ is binary. It is clear that Conditions 1 and 3 of Corollary~\ref{cor:intransconditions} are satisfied, so we need only check Condition 2. Let $H_1=\langle g\rangle$, $H_2$ a point stabiliser of $A\sqcup B$ and $H_3$ a point stabiliser of $B$, all mutually incomparable. If $H_3=\langle gz\rangle$ then $H_1\cap(H_2\cdot H_3)=1$ since then $H_2\not\in\{\langle g\rangle,\langle gz\rangle\}$, thus we may assume that $H_1Z(G)\ne H_3Z(G)$. Similarly, if $H_2Z(G)=\langle g,z\rangle$ then $H_2=\langle gz\rangle$ whence $H_1\cap(H_2\cdot H_3)=1$, and so we may assume $H_2Z(G)\ne H_1 Z(G)$. Similarly, we can assume that $H_2Z(G)\ne H_3Z(G)$. But then $H_1\cap(H_2\cdot H_3)\subseteq H_1\cap ((H_1\cap H_2)\cdot H_3)$ by Lemma~\ref{lem:intrans}, hence the result.
                                   \end{proof}
\begin{lemma}\label{lem:p^2intrans3}
    Let $G$ be a group with $Z(G)$ of order $3$. Suppose that $K\geq Z(G)$ is an elementary abelian purely binary subgroup of $G$ of order $9$. Suppose that $K=\langle z,g\rangle$ where $z\in Z(G)\setminus\{1\}$ and $g^G$ is a rational class with $|\langle g\rangle^G|\geq 4$. The action of $G$ on a $G$-set $\Omega$ is binary if and only if $\mathcal{S}(\Omega)\setminus\{\{1\},\{G\}\}$ is contained in at least one of
      $$\{\langle z\rangle^G,\langle g\rangle^G,K^G\}\setminus\{\langle gz\rangle^G\},\,\,\,\,\{\langle g\rangle^G\},\,\,\text{ or }\,\,\, \{\langle gz\rangle^G\}.$$    \end{lemma}
\begin{proof}
First, observe that since $g$ is conjugate to $g^2$ we deduce that $gz$ is conjugate to $g^2z$, whence $\langle gz\rangle$ is conjugate to $\langle gz^2\rangle=\langle g^2z\rangle$. We proceed with the following observation: if $\Omega$ is a $G$-set such that there is some $i\in \{0,1\}$ with $\{\langle gz\rangle^G,\langle g^i, z\rangle^G\}\subseteq\mathcal{S}(\Omega)$ then, taking $H_1=\langle g^2z\rangle$, $H_2=\langle gz\rangle$ and $H_3=\langle g'^i,z\rangle$ for $g'\in g^G\setminus\langle g,z\rangle$ we compute \begin{equation}\label{eq:gz}
    H_1\cap (H_2\cdot H_3)\supseteq H_1\cap K=H_1\,\,\text{ but }\,\, H_1\cap ((H_1\cap H_2)\cdot H_3)=H_1\cap H_3=1.
\end{equation}
Consequently, no binary action can simultaneously have stabilisers conjugate to $\langle gz\rangle$ and one of $Z(G)$ or $K$.

We deduce from Lemma~\ref{lem:cycstabs} and~\eqref{eq:gz},  that $\mathcal{S}(\Omega)\setminus\{\{1\},\{G\}\}$ is contained in at least one of
      $\{\langle z\rangle^G,\langle g\rangle^G,K^G\}$ or $\{\langle g\rangle^G, \langle gz\rangle^G\}.$ Thus, to prove the forward implication it suffices to show that if $A=(G:\langle gz\rangle)$ and $B=(G:\langle g\rangle)$  then the action of $G$ on $A\sqcup B$ is not binary. This is straightforward: take $H_1=\langle g\rangle$, $H_2=\langle gz\rangle$, and $H_3=\langle gz^2\rangle$ so that $H_1\cap (H_2\cdot H_3)=H_1$ but $H_1\cap ((H_1\cap H_2)\cdot H_3)=1$. This completes the proof of the forward implication when $\langle g\rangle^G\ne\langle gz\rangle^G$.
This completes the proof of the forward implication.

To prove the converse we may assume throughout that no orbit has stabiliser $1$ or $G$, and that no two orbits are equivalent by~Lemmas~\ref{GGlem2.7} and~\ref{GGlem2.6}, respectively.      
We claim that the conditions of Corollary~\ref{cor:intransconditions} are satisfied when we take $(G:K)$ to be $A$ and $(G:Z(G))$ to be $B$. Indeed, there is a unique $B$-stabiliser, and it is contained in every $A$-stabiliser hence all conditions are satisfied vacuously.     Therefore, the action of $G$ on $A\sqcup B$ is binary by Corollary~\ref{cor:intransconditions}. This completes the proof in the case that $g$ is conjugate to $gt$ for some $t\in Z(G)\setminus \{1\}$.
    
    Suppose finally that $\langle g\rangle^G\ne\langle gz\rangle^G$.  Define $A':=A\sqcup B$ (with $A$ and $B$ as defined in the previous paragraph), and $B':=(G:\langle g\rangle)$.     Arguing identically as in the penultimate paragraph of the proof of Lemma~\ref{lem:p^2intrans2}, we deduce that the action of $G$ on $A'\sqcup B'$ is binary, as desired. 
                                                  \end{proof}
\subsection{$G$ has a centre of order $6$.}

The universal cover of the simple group $\mathrm{Fi}_{22}$ has a centre with order $6$ -- this turns out to be rather tricky and so we will need to develop some extra machinery. Throughout this section we always consider groups $G$ such that $|Z(G)|=6$.

\begin{lemma}\label{lem:cent6forbidden}
    Suppose that the action of $G$ on $\Omega$ is binary, and that $Z(G)=\langle z\rangle$ has order 6. Suppose $g\in G\setminus Z(G)$ has order 2 and that $g$ is not conjugate to $gz^3$. If there are integers $i,j$ such that $\langle g,z^i\rangle^G,\langle gz^j\rangle^G\in\mathcal{S}(\Omega)$ where $\langle gz^{j}\rangle\cap\langle g,z^i\rangle\leq Z(G)$. Then $\{\langle g,z\rangle^G,\langle g,z^3\rangle^G,\langle z\rangle^G ,\langle z^3\rangle^G\}\cap \mathcal{S}(\Omega)=\emptyset$.\end{lemma}
\begin{proof}
    
    Take $H_1=\langle gz^j\rangle$, $H_2=\langle g,z^i\rangle$, and $H_3=\langle g',z^k\rangle$ where $g'\in \{1\}\cup \left(g^G\setminus\{g\}\right)$ and $k\in \{1,3\}$. Then $H_1\cap((H_1\cap H_2)\cdot H_3)\subseteq H_1\cap \langle g',z\rangle=\langle z^{2j}\rangle$, but $$H_1\cap(H_2\cdot H_3)=H_1\cap(\langle g,z^i\rangle\cdot \langle g',z^k\rangle)\supseteq H_1\cap\langle g,z^3\rangle\supseteq\langle gz^{3j}\rangle.$$ The result follows from Lemma~\ref{lem:basiccriteria}.
\end{proof}

For the remainder of this section we add the following hypotheses: Let $G$ be a group with $|Z(G)|=6$ acting on a set $\Omega$. Suppose that $K\geq Z(G)$ is an abelian purely binary subgroup of order 12, $K=\langle g,z\rangle$ where $g$ is an involution and $\langle z\rangle=Z(G)$. Finally, assume that no two subgroups of $K$ are $G$-conjugate, and that $|g^G|\geq 3$. 

\begin{lemma}\label{lem:nogzi}
If $\mathcal{S}(\Omega)\subseteq\left\{\{1\},\langle g,z\rangle^G,\langle g,z^3\rangle^G,\langle g\rangle^G,\langle gz^2\rangle^G,\langle z\rangle^G,\langle z^2\rangle^G,\langle z^3\rangle^G,\{G\}\right\}$ then the action of $G$ on $\Omega$ is binary.
\end{lemma}

\begin{proof}
              
     Start by setting $A_0:=(G:\langle z^3\rangle)\sqcup(G:\langle z^2\rangle ),$ and $B_0:=(G:\langle g\rangle)$. The action on $A_0$ is binary by Lemma~\ref{lem:centralstabs}. Moreover, distinct conjugates of $\langle g\rangle$ meet trivially, as do distinct $A_0$-stabilisers, and so Conditions 1 and 3 of Corollary~\ref{cor:intransconditions} are satisfied. Let $H_1$ be an $A_0$-stabiliser, let $H_2$ be a stabiliser of $A_0\sqcup B_0$ and $H_3$ a $B_0$-stabiliser such that the $H_i$ are mutually incomparable. Thus there is some integer $i\in \{2,3\}$ and conjugates $g'\ne g''$ of $g$ such that either $(H_1,H_2,H_3)=(\langle z^i\rangle,\langle z^{i+1}\rangle,\langle g'\rangle)$ or $(H_1,H_2,H_3)=(\langle z^i\rangle,\langle g'\rangle,\langle g''\rangle)$. In the former case, $H_1\cap (H_2\cdot H_3)=\langle z^i\rangle\cap\langle z^{i+1},g'\rangle=1,$ and in the latter case, $H_1\cap (H_2\cdot H_3)\subseteq H_1\cap((H_1\cap H_2)\cdot H_3)$ by Lemma~\ref{lem:intrans}. Therefore, the action of $G$ on $A_0\sqcup B_0$ is binary by Corollary~\ref{cor:intransconditions}.

                                Next, set $A_1:=(G:\langle g,z^3\rangle)$ and $B_1:= A_0\sqcup B_0$; it is easy to see that Conditions 1 and 3 of Corollary~\ref{cor:intransconditions} are satisfied. Take $H_1:=\langle g,z^3\rangle$, and let $H_2$ be an $(A_1\sqcup B_1)$-stabiliser, and $H_3$ a $B_1$-stabiliser such that each of the $H_i$ are mutually incomparable. Then $H_2\in\{\langle g',z^3\rangle,\langle g'\rangle,\langle z^2\rangle\}$ for some $g'\in g^G\setminus\{g\}$, and $H_3\in\{\langle z^2\rangle,\langle g''\rangle\}$ for some $g''\in g^G\setminus\{g,g'\}$. If $\langle z^2\rangle\not\in \{H_2,H_3\}$ then $H_1\cap (H_2\cdot H_3)\subseteq H_1\cap((H_1\cap H_2)\cdot H_3)$ by Lemma~\ref{lem:intrans}. On the other hand if $H_3= \langle z^2\rangle$, then $$H_1\cap ((H_1\cap H_2)\cdot H_3)=\begin{cases}
         1&\text{ if $H_2=\langle g'\rangle$}\\
         \langle z^3\rangle&\text{ if $H_2=\langle g',z^3\rangle$}
     \end{cases}=H_1\cap (H_2\cdot H_3),$$ and if $H_2=\langle z^2\rangle$ then $H_3=\langle g''\rangle$, whence $H_1\cap (H_2\cdot H_3)=1$. Therefore, the conditions of Corollary~\ref{cor:intransconditions} are met, whence the action of $G$ on $A_1\sqcup B_1$ is binary.   
     
          Now, set $A_2=(G:\langle z\rangle)$ and $B_2=A_1\sqcup B_1$. Then every pair of incomparable $B_2$-stabilisers intersects in a subgroup of $\langle z\rangle$, hence conditions 1 and 3 of Corollary~\ref{cor:intransconditions} hold. We set $H_1=\langle z\rangle$, and let $H_2$ be a $(A_2\sqcup B_2)$-stabiliser and $H_3$ a $B_2$-stabiliser such that each $H_i$ is incomparable to the others. Then $H_2,H_3\in\{\langle g,z^3\rangle ,\langle g\rangle \}^G$ with $Z(G)H_2\ne Z(G)H_3$, whence $H_1\cap (H_2\cdot H_3)\subseteq H_1\cap ((H_1\cap H_2)\cdot H_3)$ by Lemma~\ref{lem:intrans}. Therefore the action on $A_2\sqcup B_2$ is binary by Corollary~\ref{cor:intransconditions}. 

     Next, set $A_3:=(G:\langle g,z\rangle)$ and $B_3=A_2\sqcup B_2$; the intersection of any two distinct conjugates of $\langle g,z\rangle$ is $\langle z\rangle$, which contains the intersection of any two incomparable $B_3$-stabilisers, hence Conditions 1 and 3 of Corollary~\ref{cor:intransconditions} are satisfied. Let $H_1=\langle g,z\rangle$, and suppose that $H_2$ and $H_3$ are $(A_3\sqcup B_3)$-stabilisers with no two of the $H_i$ comparable. Then $H_2$ and $H_3$ are conjugates of either $\langle g\rangle$, $\langle g,z^3\rangle$, or $\langle g,z\rangle$. Observe that $Z(G)H_i$ are distinct since the $H_i$ are incomparable. Thus $H_1\cap(H_2\cdot H_3)\subseteq H_1\cap((H_1\cap H_2)\cdot H_3)$ by Lemma~\ref{lem:intrans}. Consequently, the action of $G$ on $A_3\sqcup B_3$ is binary.

  Finally, take $A_4:=(G:\langle gz^2\rangle)$, and $B_4:=A_3\sqcup B_3$; we shall now appeal to Lemma~\ref{lem:intransconditions2}. Any two incomparable $A_4$-stabilisers intersect at $\langle z^2\rangle$, and  if $L_1$ is contained in the intersection of incomparable $B_4$-stabilisers, then $L_1=\langle z^i\rangle$ for some $i\in\{0,1,2,3\}$. Thus $$L_1\cdot \langle gz^2\rangle\cap L_1\cdot\langle g'z^2\rangle=\langle z^{gcd(2,i)}\rangle=L_1\cdot (\langle gz^2\rangle \cap\langle g'z^2\rangle)$$ whenever $g'\in g^G\setminus \{g\}$, thus Condition 3 of Lemma~\ref{lem:intransconditions2} holds. Similarly, if $L_2$ is contained in the intersection of incomparable $B_4$-stabilisers then $L_2=\langle z^j\rangle$ for some $j\in\{0,1,2,3\}$, hence

    $$(L_1\cdot\langle gz^2\rangle)\cap (L_2\cdot\langle gz^2\rangle)=\begin{cases}
        \langle g,z\rangle\cap\langle gz^2\rangle=\langle gz^2\rangle&\text{if exactly one of $i,j$ is odd,}\\
        \langle g,z\rangle\cap\langle g,z\rangle=\langle g,z\rangle&\text{if both of $i,j$ are odd,}\\
        \langle gz^2\rangle\cap\langle gz^2\rangle=\langle gz^2\rangle&\text{if both of $i,j$ are even.}
    \end{cases}$$
     On the other hand
     $$(L_1\cap L_2)\cdot\langle gz^2\rangle=\langle z^{\lcm(i,j)},gz^2\rangle=\begin{cases}\langle g,z\rangle&\text{if both of $i,j$ are odd,}\\
        \langle gz^2\rangle&\text{otherwise,}\\
        
    \end{cases}$$
    and thus Condition 4 of Lemma~\ref{lem:intransconditions2} also holds. It remains to check Condition 2. Take $H_1=\langle gz^2\rangle$, and let $H_2$ and $H_3$ be $(A_4\sqcup B_4)$- and $B_4$-stabilisers, respectively such that no two are comparable. Neither $H_2$ nor $H_3$ are $\langle z^2\rangle$ since $\langle z^2\rangle\leq \langle gz^2\rangle$. Suppose first that $H_3$ is central. Then $H_3\in \{\langle z\rangle,\langle z^3\rangle\}$. Thus either $H_2\in H_1^G\cup\langle g\rangle^G$, or $H_3=\langle z\rangle$ and $H_2\in\langle g,z^3\rangle^G$. In the former case  \begin{equation*}\label{eq:gz^2}      H_1\cap (H_2\cdot H_3)=\begin{cases}
        \langle z^2\rangle&\text{if $H_2\in H_1^G$ or $H_3=\langle z\rangle$}\\
        1&\text{otherwise}
    \end{cases}=H_1\cap ((H_1\cap H_2)\cdot H_3),\end{equation*}
    and in the latter case, $H_1\cap H_2$ is either $\langle g\rangle$ or $1$, so $$H_1\cap (H_2\cdot H_3)=\begin{cases}
        H_1&\text{if $H_1\cap H_2=\langle g\rangle$}\\
        \langle z^2\rangle&\text{if $H_1\cap H_2\ne\langle g\rangle$}
    \end{cases}=H_1\cap((H_1\cap H_2)\cdot H_3).$$
    A similar argument shows that if $H_2$ is central then $H_1\cap(H_2\cdot H_3)=H_1\cap((H_1\cap H_2)\cdot H_3)$, so we assume that neither is central. 
    
    Thus $H_2\in \langle gz^2\rangle^G\cup\langle g,z\rangle^G\cup\langle g,z^3\rangle^G\cup\langle g\rangle^G$ and $H_3\in\langle g,z^3\rangle^G\cup\langle g,z\rangle^G\cup\langle g\rangle^G$. Write $H_3=\langle g_3,z^i\rangle$ where $g_3\in g^G$ and $i\in \{0,1,3\}$. If $H_2=\langle g,z^3\rangle$, then $g_3\ne g$, and so $$H_1\cap (H_2\cdot H_3)=H_1\cap (\langle g \rangle\cdot H_3)=H_1\cap((H_1\cap H_2)\cdot H_3),$$ thus we assume that $H_2\not\in\{\langle g,z^3\rangle,\langle g,z\rangle\}$, and so $Z(G)H_2\ne Z(G)H_1$. Now, if $g_3\in H_2$, then (since $H_2$ and $H_3$ are incomparable) $H_2=\langle g_3z^2\rangle$ and so $H_1\cap (H_2\cdot H_3)=\langle z^2\rangle =H_1\cap((H_1\cap H_2)\cdot H_3)$, whence we may assume that $Z(G)H_3\ne Z(G)H_2$. Finally, if $g_3=g$, then $H_3=\langle g,z^3\rangle$ and so $$H_1\cap (H_2\cdot H_3)=\begin{cases}
        H_1&\text{ if $z^2\in H_2$}\\
        \langle g\rangle&\text{otherwise}\\
    \end{cases}=H_1\cap((H_1\cap H_2)\cdot H_3),$$ and so we assume that none of the $Z(G)H_j$ are equal. Thus, $H_1\cap (H_2\cdot H_3)\subseteq H_1\cap((H_1\cap H_2)\cdot H_3)$ by Lemma~\ref{lem:intrans}. Therefore, all conditions of Lemma~\ref{lem:intransconditions2} are satisfied, and so the action of $G$ on $A_4\sqcup B_4$ is binary.

     Therefore, any action of $G$ all of whose stabilisers lie in $$\{\{1\},\langle g,z\rangle^G,\langle g,z^3\rangle^G,\langle g\rangle^G,\langle gz^2\rangle^G,\langle z\rangle^G,\langle z^2\rangle^G,\langle z^3\rangle^G,\{G\}\}$$ is binary by Lemmas~\ref{GGlem2.5} and~\ref{GGlem2.7}.
\end{proof}

\begin{lemma}\label{lem:gz}
    If $\mathcal{S}(\Omega)\subseteq\left\{\{1\},\langle g,z\rangle^G,\langle g,z^3\rangle^G,\langle gz\rangle^G,\langle gz^3\rangle^G,\langle z\rangle^G,\langle z^2\rangle^G,\langle z^3\rangle^G,\{G\}\right\}$ then the action of $G$ on $\Omega$ is binary.
\end{lemma}
\begin{proof}
    Set $A=(G:\langle gz\rangle)$ and $$B=(G:\langle g,z\rangle)\sqcup(G:\langle g,z^3\rangle)\sqcup(G:\langle z^2\rangle)\sqcup(G:\langle z^3\rangle)\sqcup(G:\langle z\rangle).$$ The action of $G$ on $B$ is binary by Lemma~\ref{lem:nogzi}. Additionally, any two incomparable $A$-stabilisers intersect at $\langle z^2\rangle$, and  if $L_1$ is contained in the intersection of incomparable $B$-stabilisers, then $L_1=\langle z^i\rangle$ for some $i$. Thus $$L_1\cdot \langle gz\rangle\cap L_1\cdot\langle g'z\rangle=\langle z^{gcd(2,i)}\rangle=L_1\cdot (\langle gz\rangle \cap\langle g'z\rangle)$$ whenever $g'\in g^G\setminus \{g\}$, thus Condition 3 of Lemma~\ref{lem:intransconditions2} holds. Similarly, if $L_2$ is contained in the intersections of incomparable $B$-stabilisers then $L_2=\langle z^j\rangle$ for some $j$, hence

    $$(L_1\cdot\langle gz\rangle)\cap (L_2\cdot\langle gz\rangle)=\begin{cases}
        \langle g,z\rangle\cap\langle gz\rangle=\langle gz\rangle&\text{if exactly one of $i,j$ is odd,}\\
        \langle g,z\rangle\cap\langle g,z\rangle=\langle g,z\rangle&\text{if both of $i,j$ are odd,}\\
        \langle gz\rangle\cap\langle gz\rangle=\langle gz\rangle&\text{if both of $i,j$ are even.}
    \end{cases}$$
     On the other hand
     $$(L_1\cap L_2)\cdot\langle gz\rangle=\langle z^{\lcm(i,j)},gz\rangle=\begin{cases}
        \langle gz\rangle&\text{if at least one of $i,j$ is even,}\\
        \langle g,z\rangle&\text{if both of $i,j$ are odd,}
    \end{cases}$$
    and thus Condition 4 of Lemma~\ref{lem:intransconditions2} also holds. It remains to check Condition 2.

    Take $H_1:=\langle gz\rangle$, and let $H_2$ and $H_3$ be $(A\sqcup B)$- and $B$-stabilisers, respectively, such that no two are comparable. Suppose $H_3$ is central. Then $H_3\in \{\langle z\rangle,\langle z^3\rangle\}$ since $\langle z^2\rangle\leq H_1$. Thus there is some $g'\in g^G\setminus\{g\}$ such that either $H_2\in \{\langle g',z^3\rangle,\langle g'z\rangle\}$, or $H_3=\langle z\rangle$ and $H_2=\langle g,z^3\rangle$. In the former case $$H_1\cap (H_2\cdot H_3)=\begin{cases}
        \langle z^2\rangle&\text{if $H_2=\langle g'z\rangle$ or $H_3=\langle z\rangle$}\\1&\text{otherwise}
    \end{cases}= H_1\cap((H_1\cap H_2)\cdot H_3)$$ and in the latter case $H_1\cap (H_2\cdot H_3)=H_1                        =H_1\cap((H_1\cap H_2)\cdot H_3).$ A similar argument shows that if $H_2$ is central then $H_1\cap(H_2\cdot H_3)=H_1\cap((H_1\cap H_2)\cdot H_3)$, so we assume that neither is central.
                        
    Since neither $H_2$ nor $H_3$ is central, we deduce that $H_2\in \langle gz\rangle^G\cup\langle g,z\rangle^G\cup\langle g,z^3\rangle ^G$ and $H_3\in\langle g,z\rangle^G\cup\langle g,z^3\rangle^G$, so we may write $H_3=\langle g_3,z^i\rangle$ where $g_3\in g^G$, and $i\in\{1,3\}$. Suppose that $H_2=\langle g_2z\rangle$ for some $g_2\in g^G\setminus\{g\}$. If $g_3=g_2$ then $$H_1\cap((H_1\cap H_2)\cdot H_3)=H_1\cap(\langle z^2\rangle \cdot \langle g_3,z^i\rangle)=H_1\cap\langle g_3,z\rangle=
        \langle z^2\rangle,
    $$
    and $H_1\cap (H_2\cdot H_3)=\langle gz\rangle\cap\langle g_2,z\rangle=\langle z^2\rangle$. On the other hand, if $g_3=g$, then $H_1\cap((H_1\cap H_2)\cdot H_3)=H_1$, thus $H_1\cap (H_2\cdot H_3)\subseteq H_1\cap((H_1\cap H_2)\cdot H_3)$. Finally, if $g_3\not\in\{g,g_2\}$, then $H_1\cap (H_2\cdot H_3)\subseteq H_1\cap((H_1\cap H_2)\cdot H_3)$ by Lemma~\ref{lem:intrans}. Thus we may write $H_2=\langle g_2,z^j\rangle$ where $g_2\in g^G$, and $j\in\{1,3\}$. Since $H_2$ and $H_3$ are incomparable, $g_2\ne g_3$. If no two of $g_1,g_2,g_3$ are equal, then $H_1\cap (H_2\cdot H_3)\subseteq H_1\cap((H_1\cap H_2)\cdot H_3)$ by Lemma~\ref{lem:intrans}; it remains to consider the situation that either $g_2=g$ or $g_3=g$. Suppose first that $g_2=g$. Then $H_2=\langle g,z^3\rangle$, and so $$H_1\cap((H_1\cap H_2)\cdot H_3)=H_1\cap (\langle gz^3\rangle\cdot\langle g_3,z^i\rangle)=H_1\cap (\langle g,z^3\rangle\cdot\langle g_3,z^i\rangle)=H_1\cap(H_2\cdot H_3).                      $$
    Finally, if $g_3=g$ and $g_2\ne g$, then $H_3=\langle g,z^3\rangle$, and so \begin{align*}
        H_1\cap((H_1\cap H_2)\cdot H_3)&=\begin{cases}
            H_1\cap (1\cdot\langle g,z^3\rangle)=\langle gz^3\rangle&\text{if $j=3$}\\
            H_1\cap(\langle z^2\rangle\cdot \langle g,z^3\rangle)=H_1&\text{if $j=1$}
        \end{cases}\\&=H_1\cap (\langle g_2,z^j\rangle\cdot\langle g,z^3\rangle)\\&=H_1\cap(H_2\cdot H_3).
    \end{align*}                      Therefore, Condition 2 of Corollary~\ref{cor:intransconditions} holds, and so the action of $G$ on $A\sqcup B$ is binary.

    Finally, set $A':=(G: \langle gz^3\rangle)$, and $B'=A\sqcup B$. Any two incomparable $A'$-stabilisers intersect trivially. If $L_1$ is contained in the intersection of two incomparable $B'$-stabilisers then either $L_1$ is central, or $L_1$ is a conjugate of $\langle gz^3\rangle$. In the former case, $$L_1\cdot \langle gz^3\rangle\cap L_1\cdot\langle g'z^3\rangle=L_1$$ for any $g'\in g^G\setminus\{g\}$, and in the latter case $$L_1\cdot \langle gz^3\rangle\cap L_1\cdot\langle g'z^3\rangle=\{1,g''z^3, gz^3,g''g\}\cap \{1,g''z^3, g'z^3,g''g'\}=L_1$$ for any $g',g''\in g^G$ with $g'\ne g$ (using that the action on $(G:\langle gz^3\rangle)$ is binary together with Lemma~\ref{lem:cycbintest} to see that $g''g'\ne gz^3$ and $g''g\ne g'z^3$). Consequently, Condition 3 of Lemma~\ref{lem:intransconditions2} holds. Let $L_2$ be contained in the intersection of two incomparable $B'$ stabilisers and let $H=\langle gz^3\rangle$. Observe that Condition 4 of Lemma~\ref{lem:intransconditions2} always holds if either $L_1$ or $L_2$ is trivial, so we may assume otherwise. Suppose first that $L_1=\langle z^i\rangle$ and $L_2=\langle z^j\rangle$ for some $i,j\in\{1,2,3\}$. Then
    $$(L_1\cdot\langle gz^3\rangle)\cap (L_2\cdot\langle gz^3\rangle)=\langle gz^3,z^i\rangle\cap \langle gz^3,z^j\rangle=\begin{cases}
        \langle g,z\rangle&\text{if $i=j=1$,}\\
        \langle g,z^3\rangle&\text{if one of $i,j$ is 3 and the other divides 3,}\\
        \langle gz\rangle&\text{if one of $i,j$ is 2 and the other divides 2,}\\
        \langle gz^3\rangle &\text{otherwise,}
    \end{cases}$$
    and 
    $$(L_1\cap L_2)\cdot \langle gz^3\rangle=\langle gz^3,z^{\lcm(i,j)}\rangle=\begin{cases}
        \langle g,z\rangle&\text{if $i=j=1$,}\\
        \langle g,z^3\rangle&\text{if one of $i,j$ is 3 and the other divides 3,}\\
        \langle gz\rangle&\text{if one of $i,j$ is 2 and the other divides 2,}\\
        \langle gz^3\rangle &\text{otherwise,}
    \end{cases}$$
    hence Condition 4 holds in this case. Suppose now that exactly one of $L_1$ and $L_2$ is conjugate to $\langle gz^3\rangle$, without loss of generality $L_1=\langle g'z^3\rangle$ for $g'\in g^G$ and $L_2=\langle z^j\rangle$ for some $j$. Then $$(L_1\cdot\langle gz^3\rangle)\cap (L_2\cdot\langle gz^3\rangle)=\{1,g'z^3, gz^3,g'g\}\cap \langle gz^3,z^j\rangle=
    \langle gz^3\rangle\subseteq (L_1\cap L_2)\cdot \langle gz^3\rangle.$$ Finally, if $L_1=\langle g'z^3\rangle$ and $L_2=\langle g''z^3\rangle$ for $g',g''\in g^G$, then 
    $$(L_1\cdot\langle gz^3\rangle)\cap (L_2\cdot\langle gz^3\rangle)=(\langle g'z^3\rangle\cdot \langle gz^3\rangle)\cap( \langle g''z^3\rangle\cdot\langle gz^3\rangle)=\begin{cases}
        \langle g'z^3\rangle\cdot \langle gz^3\rangle&\text{if $g'=g''$,}\\
        \langle gz^3\rangle &\text{otherwise}.
    \end{cases}$$
    Clearly the above is equal to $(L_1\cap L_2)\cdot \langle gz^3\rangle,$ hence Condition 4 of Lemma~\ref{lem:intransconditions2} holds for $A'\sqcup B'$. It remains to check Condition 2.

    Take $H_1=\langle gz^3\rangle$ and let $H_2$ and $H_3$ be $(A'\sqcup B')$- and $B'$-stabilisers, respectively such that no two are comparable. Suppose that $K\in\{H_2,H_3\}$ is central and let $K'$ be the unique element of $\{H_2,H_3\}\setminus\{K\}$. If $gz^3\in H_1\cap (H_2\cdot H_3)$, then $K'$ must be one of $\langle g,z\rangle$, $\langle gz^3\rangle$, $\langle g,z^3\rangle$, $\langle gz\rangle$, a contradiction since each such group contains $H_1$. Therefore, $H_1\cap (H_2\cdot H_3)=1$ so we may assume that neither $H_2$ nor $H_3$ is central.     Now, $$\langle gz^3\rangle\leq\langle g,z^3\rangle,\langle gz\rangle\leq \langle g,z\rangle,$$ so $Z(G)H_2\ne Z(G)H_1$ and $Z(G)H_3\ne Z(G)H_1$, and if $Z(G)H_2= Z(G)H_3$, then $\{H_2,H_3\}=\{\langle g',z^3\rangle,\langle g'z\rangle\}$ for some $g'\in g^G\setminus\{g\}$. But then $H_1\cap(H_2\cdot H_3)=1 $, so we assume $Z(G)H_2\ne Z(G)H_3$. Thus $H_1\cap(H_2\cdot H_3)\subseteq H_1\cap((H_1\cap H_2)\cdot H_3)$ by Lemma~\ref{lem:intrans}, hence the action of $G$ on $A'\sqcup B'$ is binary,

Therefore, any action of $G$ all of whose stabilisers lie in $$\{\{1\},\langle g,z\rangle^G,\langle g,z^3\rangle^G,\langle gz\rangle^G,\langle gz^3\rangle^G,\langle z\rangle^G,\langle z^2\rangle^G,\langle z^3\rangle^G,\{G\}\}$$ is binary by Lemmas~\ref{GGlem2.5} and~\ref{GGlem2.7}., as desired.
                                   \end{proof}

\begin{lemma}\label{lem:ggz}
    If $\mathcal{S}(\Omega)\subseteq\left\{\{1\},\langle g\rangle^G,\langle gz\rangle^G,\langle gz^2\rangle^G,\langle gz^3\rangle^G,\langle z^2\rangle^G,\{G\}\right\}$ then the action of $G$ on $\Omega$ is binary.
\end{lemma}
\begin{proof}
    Set $A_0=(G:\langle gz\rangle)$ and $$B_0=(G:\langle gz^2\rangle)\sqcup(G:\langle g\rangle)\sqcup(G:\langle z^2\rangle).$$ The action of $G$ on $B_0$ is binary by Lemma~\ref{lem:nogzi}. Any two incomparable $A_0$-stabilisers intersect at $\langle z^2\rangle$, and incomparable $B_0$ stabilisers intersect in subgroups of $\langle z^2\rangle$ so Conditions 1 and 3 of Corollary~\ref{cor:intransconditions} hold; it remains to check Condition 2.

    Take $H_1=\langle gz\rangle$, and let $H_2$ and $H_3$ be $(A_0\sqcup B_0)$- and $B_0$-stabilisers, respectively such that no two are comparable. Neither $H_2$ nor $H_3$ is central since $z^2\in\langle gz\rangle$. Thus we have $H_2\in\{\langle g_2z^2\rangle,\langle g_2z\rangle,\langle g_2\rangle\}$ and $H_3\in\{\langle g_3z^2\rangle,\langle g_3\rangle\}$ for conjugates $g_2,g_3$ of $g$.
    
    If $g_3=g$ we see that $$H_1\cap (H_2\cdot H_3)=\begin{cases}\langle z^2\rangle&\text{ if $H_3=\langle gz^2\rangle$ or $H_2\ne\langle g_2\rangle$}\\1&\text{ otherwise}\end{cases},$$ and $$H_1\cap((H_1\cap H_2)\cdot H_3)=\begin{cases}\langle z^2\rangle&\text{ if $H_3=\langle gz^2\rangle$ or $H_2\ne\langle g_2\rangle$}\\1&\text{ otherwise}\end{cases}$$    so we assume $g_3\ne g$. Similarly, if $g_2=g$ then $H_2\in\{\langle g\rangle,\langle gz^2\rangle\}$, and a similar calculation gives $H_1\cap (H_2\cdot H_3)=H_1\cap ((H_1\cap H_2)\cdot H_3)$, so we assume $g_2\ne g$. Finally, if $g_2=g_3$, then $H_2=\langle g_2z\rangle$, giving $H_1\cap (H_2\cdot H_3)=\langle z^2\rangle=H_1\cap ((H_1\cap H_2)\cdot H_3)$. Otherwise, $Z(G)H_1\ne Z(G)H_2$, $Z(G)H_1\ne Z(G)H_3$, and $Z(G)H_2\ne Z(G)H_3$, and so $H_1\cap(H_2\cdot H_3)\subseteq H_1\cap ((H_1\cap H_2)\cdot H_3)$ by Lemma~\ref{lem:intrans}. Therefore, Condition 2 of Corollary~\ref{cor:intransconditions} is satisfied, and the action of $G$ on $A_0\sqcup B_0$ is binary.
    
    Finally, set $A_1=(G:\langle gz^3\rangle)$ and $B_1:=A_0\sqcup B_0$. Any two incomparable $A_1$-stabilisers intersect trivially; Conditions 3 and 4 of Lemma~\ref{lem:intransconditions2} hold by an argument similar to that which appears in the proof of Lemma~\ref{lem:gz}, so we check Condition 2.

    Take $H_1=\langle gz^3\rangle$, and let $H_2$ and $H_3$ be $(A_1\sqcup B_1)$- and $B_1$-stabilisers, respectively such that no two are comparable. If one of $H_2$ or $H_3$ is central then the other must be a conjugate of $\langle g\rangle$ or $\langle gz^3\rangle$, and so $H_1\cap (H_2\cdot H_3)=1$, so we assume that neither is central. Thus we have $H_2\in\{\langle g_2z^3\rangle,\langle g_2z^2\rangle,\langle g_2z\rangle,\langle g_2\rangle\}$ and $H_3\in\{\langle g_3z^2\rangle,\langle g_3z\rangle,\langle g_3\rangle\}$ for conjugates $g_2,g_3\in g^G$.
    
    Since none of the $H_i$ are comparable, $H_2,H_3\not\in\{\langle gz^3\rangle,\langle gz\rangle\}$. Thus if one of $g_2,g_3$ equals $g$ and $gz^3\in H_2\cdot H_3$, then $gz^3\in\{gg'z^{i},g'gz^i\}$ for some integer $i$ and $g'\in \{g_2,g_3\}$. But then $g'\in Z(G)$, a contradiction, so $H_1\cap(H_2\cdot H_3)=1$. Thus we may assume that $g\not\in\{g_2,g_3\}$. We can similarly show that $g_2\ne g_3$. Therefore, $H_1\cap(H_2\cdot H_3)\subseteq H_1\cap((H_1\cap H_2)\cdot H_3)$ by Lemma~\ref{lem:intrans}. 
    
    Therefore, the action of $G$ on $A_1\sqcup B_1$ is binary by Lemma~\ref{lem:intransconditions2}, and the result now follows from Lemmas~\ref{GGlem2.5} and~\ref{GGlem2.7}.
\end{proof}
We now have enough to prove our final technical lemma.
\begin{lemma}\label{lem:cent6class}
Let $G$ be a group with $|Z(G)|=6$ acting on a set $\Omega$. Suppose that $K=\langle g,z\rangle$ is an abelian purely binary subgroup of $G$ of order 12, where $g$ is an involution and $\langle z\rangle=Z(G)$. Finally, assume that no two subgroups of $K$ are $G$-conjugate, and that $|g^G|\geq 3$. The action of $G$ on $\Omega$ is binary if and only if $\mathcal{S}(\Omega)$ is contained in at least one of the following:\begin{itemize}
    \item $\mathcal{S}_1:=\left\{\{1\},\langle g,z\rangle^G,\langle g,z^3\rangle^G,\langle g\rangle^G,\langle gz^2\rangle^G,\langle z\rangle^G,\langle z^2\rangle^G,\langle z^3\rangle^G,\{G\}\right\}$;
    
    \item $\mathcal{S}_2:=\left\{\{1\},\langle g,z\rangle^G,\langle g,z^3\rangle^G,\langle gz\rangle^G,\langle gz^3\rangle^G,\langle z\rangle^G,\langle z^2\rangle^G,\langle z^3\rangle^G,\{G\}\right\};
$

    \item $\mathcal{S}_3:=\left\{\{1\},\langle g\rangle^G,\langle gz\rangle^G,\langle gz^2\rangle^G,\langle gz^3\rangle^G,\langle z^2\rangle^G,\{G\}\right\}$.
\end{itemize}
\end{lemma}
\begin{proof}
    That each of these sets yield binary actions is just Lemmas~\ref{lem:nogzi},~\ref{lem:gz}, and~\ref{lem:ggz}---we shall show that there are no others.

    Suppose the action of $G$ on $\Omega$ is binary, and that $\mathcal{S}(\Omega)$ is neither a subset of $\mathcal{S}_1$ nor of $\mathcal{S}_2$. Then $\mathcal{S}(\Omega)$ contains a pair of the form $\{\langle gz^i\rangle^G,\langle gz^{j}\rangle^G\}$ where $i$ is even and $j$ is odd. Thus $\langle gz^i\rangle=\langle g,z^i\rangle$, and $\langle gz^i\rangle\cap \langle gz^{j}\rangle\leq Z(G)$. Therefore, $\mathcal{S}(\Omega)\subseteq \mathcal{S}_3$ by Lemma~\ref{lem:cent6forbidden}, hence the result.
    \end{proof}

\subsection{Proof of the main result}

We shall start by considering the simple groups, before investigating the more complicated situation when $Z(G)\ne 1$.\vspace{12pt}
\begin{proof}[Proof of Theorem~\ref{thm:main}]
    Define $\mathcal{S}(G)$ to be the set of all conjugacy classes of stabilisers $H$ such that the action of $G$ on $(G:H)$ is binary. By~Lemma~\ref{GGlem2.5}, if the action of $G$ on $\Omega$ is binary, then the group induced by $G$ on any single orbit is binary, so $\mathcal{S}(\Omega)\subseteq\mathcal{S}(G)$.
    
    Suppose first that $G$ is simple. Then $G$ has at most one non-trivial transitive binary action by Theorem~\ref{thm:maintrans}, hence the claim follows from Lemmas~\ref{GGlem2.6} and~\ref{GGlem2.7} in this case.

    Suppose next that $Z(G)\ne 1$, but that $G/Z(G)$ has no faithful transitive non-trivial binary actions. Then, by Theorem~\ref{thm:maintrans}, $\mathcal{S}(G)$ consists of $\{G\}$ and all (conjugacy classes of) subgroups of $Z(G)$ ($G/L$ has no faithful transitive binary actions for any $L\leq Z(G)$ when $G$ is a sporadic cover). Since $G$ is sporadic, $Z(G)$ is cyclic of order either $2,3,4,6$, or $12$, so by Lemma~\ref{lem:centralstabs}, the action of $G$ on any set $\Omega$ with $\mathcal{S}(\Omega)\subseteq \mathcal{S}(G)\setminus\{\{G\}\}$ is binary. The claim now follows in this case by~Lemma~\ref{GGlem2.7}. 

    The groups it remains to consider are the covers of $\mathrm{J}_2$, $\mathrm{Co}_1$, $\mathbb{B}$, $\mathrm{McL}$, and $\mathrm{Fi}_{22}$. We consider each separately.

    \textbf{Janko's group $\mathrm{J}_2$:} The universal cover $G=2.\mathrm{J}_2$ has exactly two faithful transitive non-trivial binary actions by Theorem~\ref{thm:maintrans}. These actions have stabilisers $\langle g\rangle $ and $\langle h\rangle$ where $g\in\mathrm{2c}$ and $h\in \mathrm{3a}$. Moreover, since $\mathrm{J}_2$ does not have any non-trivial transitive binary actions, the only transitive non-faithful binary actions of $G$ have stabiliser either $Z(G)$ or $G$. Therefore, if $\Omega$ is such that $\mathrm{RC}(G,\Omega)=2$, then $\mathcal{S}(\Omega)\subseteq \{\{1\},Z(G)^G,\langle g\rangle^G,\langle h\rangle^G,\{G\}\}=\mathcal{S}(G)$. Thus it suffices to show that the action of $G$ on $(G:Z(G))\sqcup (G:\langle g\rangle)\sqcup (G:\langle h\rangle)$ is binary.

    Set $A:=(G:Z(G))$ and $B:=(G:\langle g\rangle)\sqcup (G:\langle h\rangle)$. We compute using the \gap~Character Table Library~\cite{CTblLib1.3.11} that $n_G(\mathrm{2c},\mathrm{3a},\mathrm{2c})=n_G(\mathrm{3a},\mathrm{3a},\mathrm{2c})=0$, and so the action of $G$ on $B$ is binary by Corollary~\ref{cor:intranscharcalc}. Moreover, any two $B$ stabilisers meet trivially, so Conditions 1 and 3 of Corollary~\ref{cor:intransconditions} are satisfied. Let $H_1= Z(G)=\langle z\rangle$, and let $H_2$ and $H_3$ be $(A\sqcup B)$- and $B$-stabilisers, respectively, such that no two are comparable (thus $H_2\ne Z(G)$). Suppose $z\in H_2\cdot H_3$ and write $z=xy$ where $x$ and $y$ are generators of $H_2$ and $H_3$ respectively. Then $y=zx^{-1}$, and so $x,y$ must both have order 2, but $z\cdot\mathrm{2c}=\mathrm{2b}$~\cite{atlas}, a contradiction. Thus $H_1\cap (H_2\cdot H_3)=1$, hence all conditions of Corollary~\ref{cor:intransconditions} are satisfied. Therefore, the action of $G$ on $A\sqcup B$ is binary, as was to be shown.

    \textbf{Conway's group $\mathrm{Co}_1$:} The universal cover $G=2.\mathrm{Co}_1$ has a unique faithful transitive non-trivial binary action by Theorem~\ref{thm:maintrans}. The stabiliser of this action is of the form $\langle h\rangle$ where $h\in\mathrm{3a}$. Moreover, $G/Z(G)$ has no faithful transitive non-trivial binary actions, so the action of $G$ on a set $\Omega$ is binary only if $\mathcal{S}(\Omega)\subseteq\{\{1\},Z(G)^G,\langle h\rangle^G,\{G\}\}=\mathcal{S}(G)$. 

    Now, $n_G(\mathrm{2a,3a,2a})=n_G(\mathrm{3a,3a,2a})=0$,
    thus, since $Z(G)$ has prime order (generated by $z\in \mathrm{2a}$), and $\mathrm{3a}$ is real, the action of $G$ on $(G:Z(G))\sqcup(G:\langle h\rangle)$ is binary by Corollary~\ref{cor:intranscharcalc}, as claimed.

    \textbf{Baby monster group $\mathbb{B}$:} The universal cover $G=2.\mathbb{B}$ has a unique faithful transitive non-trivial binary action by Theorem~\ref{thm:maintrans}, as does the simple quotient $G/Z(G)$. Thus, if $g\in\mathrm{2b}$ then $\langle g,z\rangle$ is purely binary of order 4. The result now follows from  Lemma~\ref{lem:p^2intrans2} and the fact that $g$ is conjugate to $gz$.
    
    \textbf{McLaughlin's group $\mathrm{McL}$:} The faithful transitive non-trivial binary actions of the universal cover $G=3.\mathrm{McL}$ have stabilisers of the form $\langle h\rangle$ and $\langle hz\rangle$ where $h\in\mathrm{3c}$ and $ z\in Z(G)$ by Theorem~\ref{thm:maintrans}. Moreover, the images in the simple quotient of both classes of stabilisers give the stabiliser of the unique faithful transitive non-trivial binary action of $\mathrm{McL}$. Thus, $\langle h,z\rangle$ is purely binary of order 9. Therefore, since $\mathrm{3c}$ is real (and $h$ is not conjugate to $hz$), the result follows from  Lemma~\ref{lem:p^2intrans3}.
    
    \textbf{Fischer's group $\mathrm{Fi}_{22}$:} Suppose first that $2\mid |Z(G)|$. The group $G=Z(G).\mathrm{Fi}_{22}$ has exactly two faithful transitive non-trivial binary actions by Theorem~\ref{thm:maintrans}. Let $z\in Z(G)$. If $z\in Z(G)$ has order 2 or 6 then the images in the quotient $G/\langle z\rangle$ of stabilisers from both classes is a stabiliser of the unique faithful transitive non-trivial binary action of $G/\langle z\rangle$. And if $|Z(G)|=6$ and $z$ has order 3 then the two classes yielding binary actions descend to the corresponding two classes in $G/\langle z\rangle$. Consequently, when $2\mid |Z(G)|$ we have that $\langle g,Z(G)\rangle$ is purely binary for $g\in\mathrm{2b}$. The claim now follows from Lemma~\ref{lem:p^2intrans2} when $|Z(G)|=2$ and from Lemma~\ref{lem:cent6class} when $|Z(G)|=6$ (that no two distinct subgroups of $\langle g,Z(G)\rangle$ are conjugate is easily verified from information in the $\mathbb{ATLAS}$~\cite{atlas}). 

    Suppose now that $|Z(G)|=3$. Then both $G$ and $G/Z(G)$ have unique faithful transitive non-trivial binary actions, with the image of a corresponding stabiliser in $G$ being such a stabiliser in $G/Z(G)$. Therefore, if $g\in\mathrm{2a}$ and $\langle z\rangle =Z(G)$ then $\mathcal{S}(G)=\{\{1\},Z(G)^G,\langle g\rangle^G,\langle g,z\rangle^G,\{G\}\}$. Thus it suffices to show that if $G$ acts on $\Omega$ with $\mathcal{S}(\Omega)=\mathcal{S}(G)$, then $\mathrm{RC}(G,\Omega)=2$. First, the action of $G$ on $(G:\langle g,z\rangle)\sqcup(G:\langle z\rangle)$ is binary as the conditions of Corollary~\ref{cor:intransconditions} are satisfied vacuously. Now it remains only to show that the action of $G$ on $A\sqcup B$ is binary where $A=(G:\langle g,z\rangle)\sqcup(G:\langle z\rangle)$ and $B=(G:\langle g\rangle)$, as then the result will follow from Lemmas~\ref{GGlem2.5},~\ref{GGlem2.6} and~\ref{GGlem2.7}. Clearly, Conditions 1 and 3 of Corollary~\ref{cor:intransconditions} are satisfied, so we check Condition 2. Set $H_3=\langle g\rangle$, and let $H_1$ and $H_2$ be $A$- and $(A\sqcup B)$-stabilisers, respectively so that no two $H_i$ are comparable. Thus $H_2\in\langle g\rangle^G\cup\langle g,z\rangle^G$ and  $Z(G)H_i\ne Z(G)H_j$ for each $i\ne j$, whence $H_1\cap (H_2\cdot H_3)\subseteq H_1\cap ((H_1\cap H_2)\cdot H_3)$ by Lemma~\ref{lem:intrans}. Therefore, the action of $G$ on $A\sqcup B$ is binary as was to be shown.
\end{proof}

\bibliographystyle{plain}\bibliography{sporbibs}
\end{document}